\documentclass[11pt,fleqn]{amsart}

\usepackage{amssymb,amsmath,amsthm}

\usepackage{graphicx}
\usepackage{color}
\usepackage{thmtools}
\usepackage{thm-restate}
\usepackage[shortlabels]{enumitem}
\usepackage{mathrsfs}
\usepackage[utf8]{inputenc}
\usepackage[T1]{fontenc}
\usepackage{dsfont}
\usepackage{soul}

\usepackage{color}
\newtheorem{theorem}{Theorem}[section]
\newtheorem{theoreml}{Theorem}[section]

\newtheorem{lemma}[theorem]{Lemma}
\newtheorem{proposition}[theorem]{Proposition}

\newtheorem{question}[theorem]{Question}
\newtheorem{corollary}[theorem]{Corollary}
\newtheorem{corollaryl}[theorem]{Corollary}

\newtheorem{fact}[theorem]{Fact}

\theoremstyle{definition}

\newtheorem*{definition*}{Definition}
\newtheorem{example}[theorem]{Example}

\theoremstyle{remark}
\newtheorem{remark}[theorem]{Remark}

\numberwithin{equation}{section}

\newcommand{\frakc}{\mathfrak{c}}

\newcommand{\eps}{\varepsilon}

\newcommand{\sd}{\sigma\delta}

\newcommand{\fs}{\mathbb{F}_{\sigma}}
\newcommand{\fsd}{\mathbb{F}_{\sd}}

\newcommand{\R}{\mathbb{R}}
\newcommand{\Q}{\mathbb{Q}}

\newcommand{\F}{\mathbb{F}}

\newcommand{\E}{\mathbb{E}}
\renewcommand{\O}{\mathbb{O}}

\newcommand{\aA}{\mathcal{A}}

\newcommand{\fF}{\mathcal{F}}
\newcommand{\gG}{\mathcal{G}}

\newcommand{\sS}{\mathcal{S}}

\newcommand{\uU}{\mathcal{U}}

\newcommand{\zZ}{\mathcal{Z}}

\DeclareMathOperator{\Exh}{Exh}

\DeclareMathOperator{\Fin}{Fin}
\DeclareMathOperator{\id}{id}

\DeclareMathOperator{\intt}{int}

\newcommand{\ol}{\overline}

\newcommand{\rstr}{\restriction}

\newcommand{\sm}{\setminus}
\newcommand{\sub}{\subseteq}
\DeclareMathOperator{\supp}{supp}

\newcommand{\seq}[2]{\big\langle#1\colon\ #2\big\rangle}
\newcommand{\seqn}[1]{\big\langle#1\colon\ n\io\big\rangle}

\newcommand{\seqk}[1]{\big\langle#1\colon\ k\io\big\rangle}
\newcommand{\seql}[1]{\big\langle#1\colon\ l\io\big\rangle}

\newcommand{\seqi}[1]{\big\langle#1\colon\ i\io\big\rangle}

\newcommand{\clopen}[1]{\left[#1\right]}

\newcommand{\ctblsub}[1]{\left[#1\right]^\omega}

\newcommand{\iA}{\in\aA}

\newcommand{\io}{\in\omega}

\newcommand{\wo}{{\wp(\omega)}}
\newcommand{\bo}{{\beta\omega}}
\newcommand{\os}{\omega^*}
\newcommand{\ostar}{\os}

\newcommand{\cso}{\ctblsub{\omega}}

\newcommand{\elli}{{\ell_\infty}}

\newcommand{\Cantor}{2^\omega}

\newcommand{\noproof}{\hfill$\Box$}

\begin{document}

\title[The Josefson--Nissenzweig theorem and filters on $\omega$]{The Josefson--Nissenzweig theorem and filters on $\omega$}
\author[W. Marciszewski]{Witold Marciszewski}
\address{Institute of Mathematics and Computer Science, University of Warsaw, Warsaw, Poland.}
\email{wmarcisz@mimuw.edu.pl}
\author[D.\ Sobota]{Damian Sobota}
\address{Kurt G\"odel Research Center, Department of Mathematics, Vienna University, Vienna, Austria.}
\email{damian.sobota@univie.ac.at}
\thanks{The first author was partially supported by the NCN (National Science Centre, Poland) research grant no.\ 2020/37/B/ST1/02613. The second author was supported by the Austrian Science Fund (FWF), research grant no. I 4570-N}

%


\begin{abstract}
For a free filter $F$ on $\omega$, endow the space $N_F=\omega\cup\{p_F\}$, where $p_F\not\in\omega$, with the topology in which every element of $\omega$ is isolated whereas all open neighborhoods of $p_F$ are of the form $A\cup\{p_F\}$ for $A\in F$. Spaces of the form $N_F$ constitute the class of the simplest non-discrete Tychonoff spaces. The aim of this paper is to study them in the context of the celebrated Josefson--Nissenzweig theorem from Banach space theory. We prove, e.g., that, for a filter $F$, the space $N_F$ carries a sequence $\langle\mu_n\colon n\in\omega\rangle$ of normalized finitely supported signed measures such that $\mu_n(f)\to 0$ for every bounded continuous real-valued function $f$ on $N_F$ if and only if $F^*\le_K\mathcal{Z}$, that is, the dual ideal $F^*$ is Kat\v{e}tov below the asymptotic density ideal $\mathcal{Z}$. Consequently, we get that if $F^*\le_K\mathcal{Z}$, then: (1) if $X$ is a Tychonoff space and $N_F$ is homeomorphic to a subspace of $X$, then the space $C_p^*(X)$ of bounded continuous real-valued functions on $X$ contains a complemented copy of the space $c_0$ endowed with the pointwise topology, (2) if $K$ is a compact Hausdorff space and $N_F$ is homeomorphic to a subspace of $K$, then the Banach space $C(K)$ of continuous real-valued functions on $K$ is not a Grothendieck space. The latter result generalizes the well-known fact stating that if a compact Hausdorff space $K$ contains a non-trivial convergent sequence, then the space $C(K)$ is not Grothendieck. 
\end{abstract}


\maketitle


\section{Introduction}

We start by recalling some standard notation. Let $X$ be a Tychonoff space. By $C_p(X)$ we denote the space of all continuous real-valued functions on $X$ endowed with the pointwise topology. $C_p^*(X)$ denotes the subspace of $C_p(X)$ consisting of all bounded functions. If $\mu$ is a Borel measure on $X$ and $f\in C_p(X)$, then we write $\mu(f)=\int_Xfd\mu$. We also say that $\mu$ is finitely supported if it can be written in the form $\mu=\sum_{i=1}^n\alpha_i\delta_{x_i}$ for some distinct points $x_1,\ldots,x_n\in X$ and real numbers $\alpha_1,\ldots,\alpha_n\in\R$, $n\io$; in this case we write $\|\mu\|=\sum_{i=1}^n\big|\alpha_i\big|$ and have $\mu(f)=\sum_{i=1}^n\alpha_if\big(x_i\big)$ for every $f\in C_p(X)$ (see Section \ref{sec:prelim} for more details). 

Let $F$ be a free filter on $\omega$. Endow the set $N_F=\omega\cup\big\{p_F\big\}$, where $p_F$ is a fixed point not belonging to $\omega$, with the topology defined in the following way:
\begin{itemize}
	\item every point of $\omega$ is isolated in $N_F$, i.e. $\{n\}$ is open for every $n\io$,
	\item a set $U\sub N_F$ is an open neighborhood of $p_F$ if and only if $U=A\cup\big\{p_F\big\}$ for some $A\in F$.
\end{itemize}
It is immediate that $N_F$ is a countable non-discrete Tychonoff (in fact, normal) space.

Spaces of the form $N_F$ naturally appear in many settings, e.g. they play an essential role in $C_p$-theory, where they have been used to provide many important examples of metrizable spaces $C_p(X)$, see e.g. \cite{vMill01}. The main purpose of this paper is to study for which free filters $F$ on $\omega$ the space $N_F$ has any of the following two properties.

\begin{definition*}\label{def:jnp_bjnp}
A Tychonoff space $X$ has \textit{the Josefson--Nissenzweig property} (resp. \textit{the bounded Josefson--Nissenzweig property}), or shortly \textit{the JNP} (resp. \textit{the BJNP}), if $X$ admits \textit{a JN-sequence} (resp. \textit{a BJN-sequence}), that is, a sequence $\seqn{\mu_n}$ of finitely supported measures on $X$ such that $\big\|\mu_n\big\|=1$ for every $n\io$ and $\mu_n(f)\to0$ for every $f\in C_p(X)$ (resp. $f\in C_p^*(X)$).
\end{definition*}

As the names suggest, these two properties are closely related to the famous Josefson--Nissenzweig theorem from Banach space theory stating that for every infinite-dimensional Banach space $E$ there exists a sequence $\seqn{x_n^*}$ in the dual space $E^*$ such that $\big\|x_n^*\big\|=1$ for every $n\io$ and $x_n^*(x)\to0$ for every $x\in E$ (that is, $\seqn{x_n^*}$ is weak* null). They were introduced and studied in \cite{BKS1}, \cite{KMSZ}, \cite{KSZprod}, \cite{KSZgroth}, and \cite{MSZ}, in the context of the Separable Quotient Problem for spaces $C_p(X)$ as well as in order to investigate Grothendieck Banach spaces of the form $C(K)$ (see below). Both of the contexts, despite originating in functional analysis, have deep connections with set theory (as demonstrated e.g. in \cite{BKS0}, \cite{Bre06}, \cite{HLO87}, \cite{KS18}, \cite{DSsurv}, \cite{SZForExt}, and \cite{SZadd}). 

\medskip

For many examples of spaces with or without the (B)JNP, we refer the reader to \cite{KMSZ}, \cite{KSZprod}, and \cite{KSZgroth}. It is, however, immediate that if a space $X$ contains a non-trivial convergent sequence, then $X$ has the JNP (see Lemma \ref{lemma:basic_jnp}.(1)), the converse though does not hold (see Example \ref{example:schachermayer}). On the other hand, $\bo$, the \v{C}ech--Stone compactification of $\omega$, is an example of a space without the JNP (see Lemma \ref{lemma:basic_jnp}.(2)). The JNP and BJNP coincide, of course, in the class of pseudocompact spaces but not in the class of all Tychonoff spaces---in \cite[Example 4.2]{KMSZ} it was proved that the space $N_{F_d}$, where $F_d$ stands for the filter dual to the standard asymptotic density ideal $\zZ$ (see Example \ref{example:idealZ} for the definition), has the BJNP but not the JNP. 

\medskip

Let us describe our main results. We first obtain that, for a given free filter $F$ on $\omega$, the existence of a JN-sequence on the space $N_F$ is actually equivalent to the existence of a non-trivial convergent sequence in $N_F$ as well as to the existence of a Kat\v{e}tov reduction from the Fr\'{e}chet filter $Fr$ to $F$ (see Section \ref{sec:prelim} for necessary definitions and the beginning of Section \ref{sec:complexity} for the proof).


\begin{theoreml}\label{theorem:main_theorem_JNP}
Let $F$ be a free filter on $\omega$. Then, the following are equivalent:
\begin{enumerate}
	\item $N_F$ has the JNP,
	\item $N_F$ contains a non-trivial convergent sequence,
	\item $F\equiv_K Fr$.
\end{enumerate}
\end{theoreml}

For any non-trivial convergent sequence $\seqn{x_n\io}$ in a given space $N_F$, the sequence $\seqn{\mu_n}$ of finitely supported probability measures defined, for every $n\io$, by the equality $\mu_n=\delta_{x_n}$ satisfies the condition that $\lim_{n\to\infty}\mu_n(A)=1$ for every $A\in F$ (see Section \ref{sec:conv_seq}). A similar statement also characterizes the bounded Josefson--Nissenzweig property for spaces $N_F$ (for the proof, see Section \ref{sec:nf_bjnp_prob}).

\begin{theoreml}\label{theorem:main_theorem_BJNP_prob}
Let $F$ be a free filter on $\omega$. Then, $N_F$ has the BJNP if and only if there is a (disjointly supported) sequence $\seqn{\mu_n}$ of finitely supported probability measures on $N_F$ such that:
\begin{enumerate}
	\item 
$\supp\big(\mu_n\big)\sub\omega$ for every $n\io$,
	\item $\lim_{n\to\infty}\mu_n(A)=1$ for every $A\in F$.
\end{enumerate}
\end{theoreml}

The above measure-theoretic characterization of spaces $N_F$ having the BJNP  can be translated to the one using Kat\v{e}tov reductions of the filter $\zZ^*$ (dual to the asymptotic density ideal $\zZ$) to filters $F$, similar to the equivalence (1)$\Leftrightarrow$(3) in Theorem \ref{theorem:main_theorem_JNP}. For definitions of various classes of ideals, see Section \ref{sec:complexity}. The next theorem follows immediately from Theorems \ref{theorem:nf_bjnp_nonpath} and \ref{theorem:nf_bjnp_z}.

\begin{theoreml}\label{theorem:main_theorem_BJNP}
Let $F$ be a free filter on $\omega$. Then, the following are equivalent:
\begin{enumerate}
	\item $N_F$ has the BJNP,
	\item there is a density ideal $I$ on $\omega$ such that $F\sub I^*$,
	\item there is a non-pathological ideal $I$ on $\omega$ such that $F\sub I^*$,
	\item $F\le_K\zZ^*$.
\end{enumerate}
\end{theoreml}

As a consequence, we get that if $F$ is dual to a density ideal or to a summable ideal, then $N_F$ has the BJNP (Corollaries \ref{cor:density_bjnp} and \ref{cor:summable_bjnp}). It also follows that every non-pathological ideal is contained in some density ideal (Corollary \ref{cor:nonpath_sub_dens}). Another corollary is that every filter $F$ such that the space $N_F$ has the BJNP is contained in an $\fsd$ P-filter (Corollary \ref{cor:nf_bjnp_fsd}), so in particular $F$ is meager and of measure zero (Corollary \ref{cor:nf_bjnp_meager_meas0}). The converse to the latter two corollaries however does not hold, in Section \ref{sec:fs_no_bjnp} we describe an $\fs$ P-filter whose space $N_F$ does not have the BJNP. The existence of such an example, together with the fact that every summable ideal is $\fs$, means that neither topological properties of a free filter $F$ on $\omega$, nor topological properties of the corresponding function spaces $C_p\big(N_F\big)$ and $C^*_p\big(N_F\big)$  determine whether the space $N_F$ has the JNP or the BJNP. More precisely, if $F$ and $G$ are any uncountable $\fs$ filters, then they are homeomorphic to the space $\mathbb{Q}\times 2^\omega$ (see \cite{vE}), and all of the spaces $C_p\big(N_F\big)$, $C_p\big(N_G\big)$, $C^*_p\big(N_F\big)$, and $C^*_p\big(N_G\big)$ are homeomorphic (see \cite{DMM}).

Using Theorems \ref{theorem:main_theorem_JNP} and \ref{theorem:main_theorem_BJNP}, as well as results from \cite{GGMA16} which show the rich structure of the class of summable ideals equipped with the Kat\v{e}tov preordering, in Section \ref{sec:summable} we indicate that the classes of those filters $F$ on $\omega$ for which the spaces $N_F$ have the JNP or the BJNP are also large.


\begin{corollaryl}\label{cor:main_corollary_continuum}
There are families $\fF_1$ and $\fF_2$, each consisting of continuum many pairwise non-isomorphic free $\fs$ P-filters on $\omega$, such that:
\begin{enumerate}[(A)]
	\item for every $F\in\fF_1$ the space $N_F$ has the JNP,
	\item for every $F\in\fF_2$ the space $N_F$ has the BJNP but does not have the JNP.
\end{enumerate}
\end{corollaryl}

\begin{corollaryl}\label{cor:main_corollary_2continuum}
There exist families $\fF_3$,  $\fF_4$, and $\fF_5$, each consisting of $2^{\frakc}$ many pairwise non-isomorphic free filters on $\omega$, such that:
\begin{enumerate}[(a)]
	\item for every $F\in\fF_3$ the space $N_F$ has the JNP,
	\item for every $F\in\fF_4$ the space $N_F$ has the BJNP but does not have the JNP,
	\item for every $F\in\fF_5$ the space $N_F$  does not have the BJNP.
\end{enumerate}
\end{corollaryl}

We also provide some applications of our results to analysis. The first one is related to the famous long-standing open Separable Quotient Problem (for Banach spaces), asking whether every infinite-dimensional Banach space admits a separable infinite-dimensional quotient. The problem has a positive answer in the case of Banach spaces $C(K)$ of continuous real-valued functions on infinite compact spaces $K$ endowed with the supremum norm. Hence, it is natural to pose its $C_p$-analogon and ask whether for every infinite space $X$ the space $C_p(X)$ admits a separable infinite-dimensional quotient. This version of the problem has recently gained much attention, see e.g. \cite{BKS0}, \cite{BKS1}, \cite{KS18}, and is also open. 
Using a result of \cite{KMSZ}, Theorem \ref{theorem:main_theorem_BJNP}, and Corollary \ref{cor:main_corollary_2continuum}, we obtain the following sufficient condition for spaces $C_p^*(X)$ (and hence for spaces $C_p(X)$ for $X$ pseudocompact) to contain a complemented copy of the space $(c_0)_p=\big\{x\in\R^\omega\colon\ x(n)\to0\big\}$, endowed with the pointwise topology, as well as we construct a large class of ``minimal'' spaces $Y$ such that, for a given space $X$, the space $C_p^*(X)$ has a complemented copy of $(c_0)_p$ provided that $Y$ embeds into $X$ (see Section \ref{sec:tychonoff} for details on arguments).

\begin{corollaryl}\label{cor:corollary_C}
Let $F$ be a free filter on $\omega$ such that $F\le_K\zZ^*$. If $X$ is a space such that $N_F$ homeomorphically embeds into $X$, then $C_p^*(X)$  contains a complemented copy of the space $(c_0)_p$.
\end{corollaryl}

The second main application of our results concerns Banach spaces which are Grothendieck. Recall that a Banach space $E$ is \textit{Grothendieck} (or has \textit{the Grothendieck property}) if every weak* null sequence in the dual space $E^*$ is also weakly null. The class of Grothendieck spaces contains such important objects as reflexive spaces, the space $\elli$ of bounded sequences or, more generally, spaces $C(K)$ for $K$ extremely disconnected (\cite{Gro53}), von Neumann algebras (\cite{Pfi94}), the space $H^\infty$ of bounded analytic functions on the open unit disc (\cite{Bou83}), etc. It is however an open problem of characterizing  compact spaces $K$ such that the space $C(K)$ is Grothendieck (see \cite[Section 3]{Die73}). It is well-known that if a compact space $K$ contains a non-trivial convergent sequence, then the space $C(K)$ is not Grothendieck. The latter fact may be stated in the following filter-like way: if the space $N_{Fr}$, where $Fr$ is the Fr\'{e}chet filter on $\omega$, homeomorphically embeds into a compact space $K$, then $C(K)$ is not a Grothendieck space. Using the results provided above, we generalize this folklore fact in the following way (for details, see Section \ref{sec:tychonoff}).

\begin{corollaryl}\label{cor:corollary_E}
Let $F$ be a free filter on $\omega$ such that $F\le_K\zZ^*$. If $K$ is a compact space such that $N_F$ homeomorphically embeds into $K$, then $C(K)$ is not a Grothendieck space.
\end{corollaryl}

Combining Theorems \ref{theorem:main_theorem_JNP} and \ref{theorem:main_theorem_BJNP} together with Corollaries \ref{cor:main_corollary_2continuum}.(b), \ref{cor:corollary_C}, and \ref{cor:corollary_E}, we obtain the following result.

\begin{corollaryl}\label{cor:corollary_F}
There exists a family $\fF$ of $2^\frakc$ many pairwise non-homeomorphic countable infinite spaces with exactly one non-isolated point and without any non-trivial convergent sequences, and such that, for every space $Y\in\fF$, if $Y$ homeomorphically embeds into a space $X$, then $C_p^*(X)$  contains a complemented copy of the space $(c_0)_p$ and if $X$ is compact, then the Banach space $C(X)$ is not a Grothendieck space. 
\end{corollaryl}

\medskip

The paper is organized as follows. Section \ref{sec:prelim} contains the description of our notation and terminology. In Section \ref{sec:af_sf_nf} we recall basic topological facts concerning spaces of the form $N_F$ and their \v{C}ech--Stone compactifications $S_F=\beta\big(N_F\big)$. 
Section \ref{sec:conv_seq} is devoted to the study for which filters $F$ the spaces $N_F$ and $S_F$ contain non-trivial convergent sequences. In Section \ref{sec:nf_char_jnp_bjnp} we provide several reformulations and characterizations of the (bounded) Josefson--Nissenzweig property for spaces of the form $N_F$ and $S_F$. In Section \ref{sec:complexity} we investigate the Kat\v{e}tov preordering on the class of those filters $F$ for which the spaces $N_F$ have the BJNP. Section \ref{sec:tychonoff} is devoted to applications of obtained results to general Tychonoff spaces and Grothendieck $C(K)$-spaces.

\subsection*{Acknowledgements}

The authors would like to thank Piotr Borodulin-Nadzieja, Piotr Koszmider, Arturo Mart\'{\i}nez-Celis, Grzegorz Plebanek, Jacek Tryba, and Lyubomyr Zdomskyy, for providing helpful comments and ideas which allowed the authors to obtain results presented in this paper. The authors are also very grateful to the anonymous referee whose remarks and suggestions led to a significant improvement in the presentation of the results.

\section{Preliminaries\label{sec:prelim}}

By $\omega$ we denote the first infinite (countable) cardinal number. By $\frakc$ we denote the continuum, i.e. the cardinality of the real line $\R$. For every $k,n\io$ such that $k<n$ we set $[k,n]=\{k,k+1,\ldots,n\}$ and $[k,n)=[k,n]\sm\{n\}$. 

If $X$ is a set, then $|X|$ denotes the cardinality of $X$ and $\wp(X)$ denotes the family of all subsets of $X$. 
As usual, $\id_X$ denotes the identity function on $X$. The complement $X\sm Y$ of a subset $Y$ of $X$ is also denoted by $Y^c$. For a family $S\sub\wp(X)$ we put $S^*=\{Y\sub X\colon X\sm Y\in S\}$; $S^*$ is said to be \textit{dual} to $S$. If $A$ is a subset of $X$, then we also put $S\rstr A=\{B\cap A\colon B\in S\}$. For two sets $A$ and $B$ the relation $A\sub^*B$ means that the difference $A\sm B$ is finite.

If $A$ is a non-empty set, then a family $F\sub\wp(A)$ is \textit{a filter on $A$} if $\emptyset\not\in F$, $A\in F$, $\big|\bigcap F\big|\le 1$, and $F$ is closed under finite intersections and taking supersets. A family $I\sub\wp(A)$ is \textit{an ideal on $A$} if $I^*$ is a filter on $A$. If $A$ is infinite, then by $Fr(A)$ we denote \textit{the Fr\'echet filter} on $A$, i.e. $Fr(A)=\{B\in\wp(A)\colon A\sm B\text{ is finite}\}$. If $A=\omega$, then we simply write $Fr=Fr(\omega)$. If $A\sub B$ are both infinite sets, then $Fr(A,B)$ denotes the filter $\{C\in\wp(B)\colon A\sm C\text{ is finite}\}$ on $B$. We have $Fr=Fr(\omega,\omega)$. The ideal $Fr^*$, dual to $Fr$ and containing only finite subsets of $\omega$, will be denoted by $Fin$. 
$\cso$ denotes the family of all infinite subsets of $\omega$.

If $F$ is a filter on a set $A$, then $F$ is \textit{free} if $\bigcap F=\emptyset$, and $F$ is \textit{principal} if $\bigcap F$ is a singleton and $\bigcap F\in F$. Note that, for a filter $F$ on $\omega$, $F$ is free if and only if $Fr\sub F$ if and only if $Fin\sub F^*$. $F$ is \textit{an ultrafilter on $A$} if $F$ is a maximal filter (with respect to inclusion) or, equivalently, if for every $B\in\wp(A)$ either $B\in F$ or $B^c\in F$.

A filter $F$ on $\omega$ is \textit{a P-filter} if for every sequence $\seqn{A_n\in F}$ there is $A\in F$ such that $A\sub^* A_n$ for every $n\io$. An ideal $I$ on $\omega$ is \textit{a P-ideal} if its dual filter $I^*$ is a P-filter, that is, if for every sequence $\seqn{A_n\in I}$ there is $A\in I$ such that $A_n\sub^*A$ for every $n\io$. 

\medskip

If $f\colon\omega\to\omega$ is a function and $\aA\sub\wo$, then we set $f(\aA)=\big\{A\in\wo\colon\ f^{-1}[A]\in\aA\big\}$. Let $F$ and $G$ be free filters on $\omega$. We say that $F$ is \textit{Kat\v{e}tov below} $G$, denoting $F\le_K G$, if there is a function $f\colon\omega\to\omega$ such that $F\sub f(G)$. Note that if $G\neq Fr$ (that is, there exists co-infinite $A\in G$), then we may assume that $f$ is a surjection. If $F\le_K G$ and $G\le_K F$, then we say that $F$ and $G$ are \textit{Kat\v{e}tov equivalent}, denoting $F\equiv_K G$. Note that if $F\sub G$, then $F\le_K G$ and that fact is witnessed by the identity function on $\omega$. Thus, the Fr\'echet filter $Fr$ is Kat\v{e}tov below any free filter $G$.

Similarly, we say that $F$ is \textit{Rudin--Keisler below} $G$, denoting $F\le_{RK} G$, if there is a function $f\colon\omega\to\omega$ such that $f(G)=F$. Obviously, if $F\le_{RK} G$, then $F\le_K G$. If $F\le_{RK} G$ and $G\le_{RK} F$, then we say that $F$ and $G$ are \textit{Rudin--Keisler equivalent}, denoting $F\equiv_{RK} G$. 

We say that filters $F$ and $G$ are \textit{isomorphic} if there is a bijection $f\colon\omega\to\omega$ such that $f(G)=F$. Obviously, if $F$ and $G$ are isomorphic, then $F\equiv_{RK} G$; if $F$ and $G$ are ultrafilters, then the latter implication can be reversed.

We also apply the above nomenclature in the natural way to ideals (via their dual filters).

\medskip

Throughout the paper we assume that every topological space considered by us is \textit{Tychonoff}. Let $X$ be a space. A sequence $\seqn{x_n}$ of points in $X$ converging to some $x\in X$ is \textit{non-trivial} if $x_n\neq x_m$ for every $n\neq m\io$ and $x_n\neq x$ for every $n\io$. For a subset $Y\sub X$, its closure in $X$ is denoted by $\ol{Y}^X$. $Y$ is \textit{a P-set} in $X$ if the intersection of countably many open sets containing $Y$ contains $Y$ in its interior. A point $x\in X$ is \textit{a P-point} in $X$ if the singleton $\{x\}$ is a P-set in $X$. $\beta X$ denotes the \v{C}ech--Stone compactification of $X$. As usual, we write shortly $\ostar=\bo\sm\omega$. 

Given two spaces $X$ and $Y$, $X\cong Y$ denotes that they are homeomorphic.

If $X$ is a space, then by $C(X)$ and $C^*(X)$ we denote the set of all continuous real-valued functions on $X$ and the set of all bounded continuous real-valued functions on $X$, respectively. A subspace $Y$ of $X$ is \textit{$C$-embedded} (resp. \textit{$C^*$-embedded}) in $X$ if for every function $f\in C(Y)$ (resp. $f\in C^*(Y)$) there is a function $f'\in C(X)$ (resp. $f'\in C^*(X)$) such that $f=f'\rstr Y$. $C_p(X)$ and $C_p^*(X)$ respectively denote the spaces $C(X)$ and $C^*(X)$ endowed with the pointwise topology. 

We denote the Cantor space endowed with its standard topology by $\Cantor$. When we speak about measurability properties of subsets of $\Cantor$, then we always mean \textit{the standard product measure} on $\Cantor$. Each $A\in\wo$ can be associated with its characteristic function $\chi_A\colon\omega\to\{0,1\}$ and hence $A$ may be treated as an element of $\Cantor$. Thus, every filter or ideal on $\omega$ may be considered as a subset of $2^\omega$ and so we can talk about its topological and measure-theoretic features like meagerness, measurability, Borel complexity, etc. Similar remarks of course also hold for any arbitrary set $X$, the product space $2^X$, and subfamilies of the power set $\wp(X)$.

\medskip

Let $\aA$ be a Boolean algebra. A family $\uU\sub\aA$ is \textit{an ultrafilter} on $\aA$ if it is a maximal family (with respect to inclusion) such that $0_\aA\not\in\uU$, for every $A,B\in\uU$ we have $A\wedge B\in\uU$, and for every $A,B\in\aA$ if $A\in\uU$ and $A\le B$, then $B\in\uU$. $St(\aA)$ denotes the Stone space of $\aA$, i.e. the standard totally disconnected compact space of all ultrafilters on $\aA$. If $A\iA$, then $\clopen{A}_\aA$ denotes the clopen subset of $St(\aA)$ corresponding via the Stone duality to $A$. Note that $\wo$ is a Boolean algebra when endowed with the standard set-theoretic operations. Recall that $St(\wo)\cong\bo$ and $St(\wo/Fin)\cong\ostar$. For every $A\in\wo$ we will write $\clopen{A}_\omega$ instead of $\clopen{A}_{\wo}$ and $\clopen{A}_\omega^*$ instead of $\clopen{A}_{\wo/Fin}$ for the corresponding clopen subsets of $\bo$ and $\ostar$, respectively. Of course, $\clopen{A}_\omega^*=\clopen{A}_\omega\sm\omega$ for all $A\in\wo$. Note also that a free P-filter on $\omega$ which is an ultrafilter on $\wo$ is a P-point in $\ostar$.

\medskip

Let $X$ be a space. By $Bor(X)$ we denote the Borel $\sigma$-field of $X$. If $\mu\colon Bor(X)\to\R$ is a $\sigma$-additive regular signed measure which has bounded \textit{total variation}, that is, $\|\mu\|=\sup\big\{|\mu(A)|+|\mu(B)|\colon\ A,B\in Bor(X), A\cap B=\emptyset\big\}<\infty$, then we simply say that $\mu$ is \textit{a measure on $X$}. By $|\mu|(\cdot)$ we denote \textit{the variation} of $\mu$; note that $\|\mu\|=|\mu|(X)$. $\supp(\mu)$ denotes \textit{the support} of $\mu$. If $A\in Bor(X)$, then the measure $\mu\rstr A$ on $X$ is defined by the formula $(\mu\rstr A)(B)=\mu(A\cap B)$ for every $B\in Bor(X)$. $\mu$ is \textit{a probability measure} if $\mu(A)\ge0$ for every $A\in Bor(X)$ and $\mu(X)=1$.

If $X$ is a space and $x\in X$, then $\delta_x$ denotes the one-point measure on $X$ concentrated at $x$. We say that a measure $\mu$ on $X$ is \textit{finitely supported} if there is a sequence $x_1,\ldots,x_n$ of mutually distinct points of $X$ and a sequence $\alpha_1,\ldots,\alpha_n\in\R\sm\{0\}$ such that $\mu=\sum_{i=1}^n\alpha_i\delta_{x_i}$. It follows that $\supp(\mu)=\big\{x_1,\ldots,x_n\big\}$ and that $\|\mu\|=\sum_{i=1}^n\big|\alpha_i\big|$. A sequence $\seqn{\mu_n}$ of measures on $X$ is \textit{disjointly supported} if $\supp\big(\mu_k\big)\cap\supp\big(\mu_n\big)=\emptyset$ for every $k\neq n\io$.

If $\mu$ is a measure on a space $X$ and $f\in C(X)$, then we set $\mu(f)=\int_X fd\mu$. If $\mu$ is finitely supported, then for any $f\in C(X)$ we have $\mu(f)=\sum_{x\in\supp(\mu)}\mu(\{x\})\cdot f(x)$.

\section{Spaces related to filters on $\omega$\label{sec:af_sf_nf}}

At the beginning of Introduction we recalled the standard way of assigning to a filter $F$ on $\omega$ the simple non-discrete countable Tychonoff space $N_F$. We now introduce two other objects associated with a filter $F$, the Boolean algebra $\aA_F$ and its Stone space $S_F$, and recall their relations with $N_F$. The results presented in this section are mostly standard and fairly easy, therefore we will omit their proofs.

Let $F$ be a free filter on $\omega$. By $\aA_F$ we denote the following Boolean subalgebra of $\wo$:
\[\aA_F=\big\{A\in\wo\colon\ A\in F\text{ or }A^c\in F\big\}.\]
Note that $F$ is an ultrafilter in $\aA_F$ and $\aA_F=F\cup F^*$. Put $S_F=St\big(\aA_F\big)$. Trivially, $\aA_F=\wo$ if and only if $F$ is an ultrafilter in $\wo$. For every $A\in\aA_F$, by $\clopen{A}_F$ we will denote the clopen subset of $S_F$ corresponding via the Stone duality to the element $A$ of $\aA_F$. Let also $\pi_F\colon\bo\to S_F$ denote the canonical continuous surjection defined for every ultrafilter $x\in\bo$ by the formula $\pi_F(x)=x\cap\aA_F$.

Note that $S_F$ contains a countable discrete dense subspace, consisting of all isolated points of $S_F$, which we can identify in a natural way with $\omega$, and therefore we can put $S_F^*=S_F\sm\omega$. Also, $Fin$ is an ideal in $\aA_F$ and one can show that $S_F^*$ is homeomorphic to the Stone space $St\big(\aA_F/Fin)$ of the quotient Boolean algebra $\aA_F/Fin$, or, equivalently, that the Boolean algebra of clopen subsets of $S_F^*$ is isomorphic to $\aA_F/Fin$. For every $A\in\wo$ we will also simply write $A$ for the corresponding subset of $\omega\sub S_F$. This way, if $A\in\aA_F$, then $\ol{A}^{S_F}=\clopen{A}_F$, and, conversely, if $A\in\wo$ is such that $\ol{A}^{S_F}$ is clopen, then $A\in\aA_F$ (and hence $\ol{A}^{S_F}=\clopen{A}_F$). For $A\in\aA_F$ we also write $\clopen{A}_F^*=\clopen{A}_F\cap S_F^*=\clopen{A}_F\sm\omega$; note that $\clopen{A}_F^*$ is a clopen in $S_F^*$.

There is also a special unique point $p_F\in S_F$ such that, for every $A\in\aA_F$, $p_F\in\clopen{A}_F$ if and only if $A\in F$. Formally, of course, $F=p_F$, but to focus the attention we will use the symbol $F$ when talking about the filter on $\omega$, and $p_F$ when talking about the point in $S_F$. 

\begin{lemma}\label{lemma:sf_frechet_bo}
Let $F$ be a free filter on $\omega$. Then,
\begin{enumerate}
	\item If $F=Fr$, then $S_F^*=\big\{p_F\big\}$. 
	\item If $F\neq Fr$, then $S_F^*\sm\big\{p_F\}\neq\emptyset$ and every point of $S_F^*\sm\big\{p_F\big\}$ has a clopen neighborhood in $S_F$ homeomorphic to $\bo$ and a clopen neighborhood in $S_F^*$ homeomorphic to $\omega^*$.\noproof
\end{enumerate}
\end{lemma}
%

A general relation between $\bo$ and spaces $S_F$ is described by the following proposition. Intuitively, $S_F$ is made from $\bo$ by gluing together all the ultrafilters in $\bo$ extending $F$.


\begin{proposition}\label{prop:sf_comes_from_bo}
Let $F$ be a free filter on $\omega$ and $\fF$ denote the subset of $\bo$ consisting of all ultrafilters $x\in\bo$ such that $F\sub x$. 
\begin{enumerate}
	\item $\fF$ is closed in $\bo$, $\fF\sub\ostar$, and the mapping $\varphi\colon\bo/\fF\to S_F$ given for every $x\in\bo$ by the formula $\varphi\big([x]_{\fF}\big)=x\cap\aA_F$ ($=\pi_F(x)$), where $[x]_{\fF}$ denotes the equivalence class of $x$ in the quotient space $\bo/\fF$, is a homeomorphism. Moreover, if $\pi\colon\bo\to\bo/\fF$ is the canonical quotient map, then $\varphi\circ\pi=\pi_F$ and $\pi^{-1}\big(\varphi^{-1}\big(p_F\big)\big)=\fF$.
	\item For every $A\in\cso$, the following are equivalent:
\begin{enumerate}
	\item $A\sm B$ is finite for every $B\in F$,
	\item $\clopen{A}_\omega^*\sub\fF$.
	\end{enumerate}
	\item The following are equivalent:
\begin{enumerate}
	\item $\fF$ is a P-set (in $\omega^*$),
	\item $F$ is a P-filter on $\omega$,
	\item $F$ is a maximal P-filter on $\aA_F$,
	\item $p_F$ is a P-point in $S_F^*$.\noproof
\end{enumerate}
\end{enumerate}
\end{proposition}

The following result complements Proposition \ref{prop:sf_comes_from_bo}.(1), yielding a correspondence between filters on $\omega$, closed subsets of $\ostar$, and spaces of the form $S_F$.

\begin{proposition}\label{prop:sf_comes_from_bo_converse}
For every non-empty closed subset $\fF$ of $\omega^*$, the intersection $F=\bigcap\fF$ is a free filter on $\omega$ such that the spaces $\bo/\fF$ and $S_F$ are homeomorphic.\noproof
\end{proposition}

\begin{corollary}\label{cor:correspondence}
There is a natural (many-to-one) correspondence between non-empty closed subsets of $\omega^*$ and spaces of the form $S_F$.\noproof
\end{corollary}

Note that the correspondence given by Propositions \ref{prop:sf_comes_from_bo}.(1) and \ref{prop:sf_comes_from_bo_converse} is not one-to-one, as for every two ultrafilters $x\neq y\in\omega^*$ we still have $\aA_x=\aA_y$ and so $S_x=S_y$. %
%

\begin{proposition}\label{prop:sf_bo_ultrafilter}
For every free filter $F$ on $\omega$, $S_F\cong\bo$ if and only if $F$ is an ultrafilter.\noproof
\end{proposition}
%

\begin{remark}
The implication from left to right in Proposition \ref{prop:sf_bo_ultrafilter} does not hold anymore if one exchanges $\bo$ and $S_F$ with $\ostar$ and $S_F^*$, respectively, that is, there is a filter $F$ on $\omega$ such that $S_F^*$ is homeomorphic to $\ostar$, but $F$ is not maximal (consider the filter $F=\bigcap\fF$ where $\fF=\big(\ol{A}^{\bo}\sm A\big)\cup\{x\}$ for a co-infinite set $A\in\cso$ and an ultrafilter $x\in\ostar$ such that $A\not\in x$).
\end{remark}

\begin{remark}
If $F$ is a P-filter on $\omega$ such that $F\rstr A\neq Fr(A)$ for every $A\in\cso$ (e.g. $F$ is the filter $F_d$, see Example \ref{example:idealZ}), then the answer whether the space $S_F^*$ is homeomorphic to $\ostar$ depends on the assumed system of axioms: assuming the Continuum Hypothesis, with an aid of results of \cite[Sections 1.1--1.4]{vMillHBK}, one can show that $S_F^*\cong\ostar$, whereas in any model of set theory without P-points in $\ostar$ (e.g. in the Silver model, see \cite{CG19}), one has $S_F^*\not\cong\ostar$.


\end{remark}

\medskip

For every free filter $F$ on $\omega$, we distinguish a special countable subspace of $S_F$, already presented in Introduction: 
\[N_F=\omega\cup\big\{p_F\big\}.\]
It is immediate that the topology of $N_F$ inherited from $S_F$ coincides with the topology described in Introduction, that is, every point of $\omega$ is isolated in $N_F$ (as it is isolated in $S_F$) and the family consisting exactly of all (clopen) sets of the form $A\cup\big\{p_F\big\}$, where $A\in F$, is a local base of $p_F$.
Note that every open neighborhood of $p_F$ in $N_F$ is a clopen subset of $N_F$. It follows from Lemma \ref{lemma:sf_frechet_bo}, that $N_F=S_F$ if and only if $F=Fr$. For more information on spaces of the form $N_F$ see \cite[Section 4M]{GJ60} and \cite{DMM}, where properties of associated spaces of functions were studied.

The main topological relations between the spaces $N_F$ and $S_F$ are described by the following lemma, which can be proved by appealing to \cite[Theorem 6.4, Problem 4M.1]{GJ60} (for (1)) and \cite[Section 6.9, page 89]{GJ60} (for (2)).

\begin{lemma}\label{lemma:nf_cstar_embedded_sf}
For every free filter $F$, the following hold:
\begin{enumerate}
	\item the space $N_F$ is $C^*$-embedded in $S_F$,
	\item $S_F=\beta\big(N_F\big)$, i.e. $S_F$ is the \v{C}ech--Stone compactification of $N_F$. \noproof
\end{enumerate}
\end{lemma}
%

\begin{lemma}
Let $F$ be a free filter on $\omega$. For every $A\in F$, $\ol{A}^{N_F}\cong N_F$ if and only if $\omega\sm A$ is finite or $Fr(A,\omega)\subsetneq F$. \noproof
\end{lemma}

\medskip

We now briefly recall standard facts concerning relations between spaces $N_F$ and $N_G$ (resp. $S_F$ and $S_G$) in the case when $F$ is Kat\v{e}tov below $G$. So fix free filters $F$ and $G$ on $\omega$. For every function $f\colon\omega\to\omega$, we define the mapping $\varphi_f\colon N_G\to N_F$ as follows: $\varphi_f\big(p_G\big)=p_F$ and $\varphi_f(n)=f(n)$ for every $n\io$. It follows immediately that if $F\sub f(G)$ (so $F\le_K G$), then $\varphi_f$ is continuous and $\varphi_f^{-1}\big(p_F\big)=\big\{p_G\big\}$. Conversely, if $\varphi\colon N_G\to N_F$ is a continuous mapping such that $\varphi^{-1}\big(p_F\big)=\big\{p_G\big\}$, then the function $f_\varphi\colon\omega\to\omega$ defined as $f_\varphi=\varphi\rstr\omega$ is such that $F\sub f(G)$ (so $F\le_K G$).

\begin{proposition}\label{prop:ordering_nf}
Let $F$ and $G$ be free filters on $\omega$ and $f\colon\omega\to\omega$ a function.
\begin{enumerate}
	\item If $F\sub f(G)$ 
	and $f$ is a surjection, then $\varphi_f$ maps continuously $N_G$ onto $N_F$. In particular, if $F\le_K G$ and $G\neq Fr$, then $N_G$ can be continuously mapped onto $N_F$.
	\item $F=f(G)$ and $f$ is a bijection 
	if and only if $\varphi_f\colon N_G\to N_F$ is a homeomorphism.
	\item $F\sub G$ if and only if $\varphi_{\id_\omega}$ maps continuously $N_G$ onto $N_F$. \noproof
\end{enumerate}
\end{proposition}
%

Regarding Proposition \ref{prop:ordering_nf}.(2), note that it is not true that the equivalence $F\equiv_{RK}G$ alone implies that $N_F$ and $N_G$ are homeomorphic. To see this, fix a co-infinite set $A\in\cso$ and let $G=Fr$ and $F=Fr(A,\omega)$. Then, $F\equiv_{RK}G$ but $N_G$ and $N_F$ are not homeomorphic, since $N_G$ is compact and $N_F$ is not. It follows in this case that $S_G$ and $S_F$ are not homeomorphic as well, or even that $S_F$ is not a continuous image of $S_G$, since $S_G$ ($=N_G$) is countable and $S_F$ contains $\bo$.

We also have an analogon of Proposition \ref{prop:ordering_nf} for spaces $S_F$ and $S_F^*$.

\begin{proposition}\label{prop:ordering_sf}
Let $F$ and $G$ be free filters on $\omega$.
\begin{enumerate}
	\item If $F\le_K G$ and $G\neq Fr$, then $S_F$ is a continuous image of $S_G$.
	\item If there is a bijection $f\colon\omega\to\omega$ such that $F=f(G)$,  
	then $S_F\cong S_G$ and $S_F^*\cong S_G^*$.
	\item If $F\sub G$, then $S_F$ is a continuous image of $S_G$. \noproof
\end{enumerate}
\end{proposition}

Note that if $F$ and $G$ are two free ultrafilters on $\omega$ which are not Rudin--Keisler equivalent, or even not comparable in the sense of Kat\v{e}tov, then still $S_F=\bo=S_G$, so the converse statements to Propositions \ref{prop:ordering_sf}.(1)--(3) do not hold. By Proposition \ref{prop:ordering_nf}.(2), it follows also that two non-homeomorphic spaces $N_F$ and $N_G$ may still have the \v{C}ech--Stone compactifications which are homeomorphic, i.e., $S_F\cong S_G$.


\section{Non-trivial convergent sequences in spaces $S_F$\label{sec:conv_seq}}

We study in this section for which filters $F$ the spaces $N_F$, $S_F$, and $S_F^*$ contain non-trivial convergent sequences. We start with the following simple, but useful, consequence of Lemma \ref{lemma:sf_frechet_bo} yielding that there maybe at most one point in $S_F$, namely $p_F$, being the limit of a non-trivial convergent sequence.

\begin{lemma}\label{lemma:sf_convseq_pf}
Let $F$ be a free filter on $\omega$ such that the space $S_F$ contains a non-trivial convergent sequence $\seqn{x_n}$. Then, $p_F=\lim_{n\to\infty}x_n$. \noproof
\end{lemma}


The above lemma easily implies the following result.

\begin{lemma}\label{lemma:nf_convseq_frechet}
Let $F$ be a free filter on $\omega$. Let $\seqn{x_n}$ be sequence in $N_F$ such that $x_n\neq x_m$ for every $n\neq m\io$. Then, $\seqn{x_n}$ is convergent (to $p_F$) if and only if for every $A\in F$ there is $N\io$ such that $x_n\in A$ for every $n>N$. \noproof
%
\end{lemma}
%


%

Recall that an ideal $I$ is \textit{tall} (or \textit{dense}) if for every $A\in\cso$ there is $B\in I\cap\cso$ contained in $A$. Trivially, every maximal ideal  is tall but $Fin$ itself is not tall. As a consequence of Lemma \ref{lemma:nf_convseq_frechet}, with an aid of Proposition \ref{prop:sf_comes_from_bo}.(2), we obtain the following characterization of the existence of non-trivial convergent sequences in spaces $N_F$.

\begin{proposition}\label{prop:nf_convseq_char}
Let $F$ be a free filter on $\omega$. Let $\fF$ denote the closed subset of $\ostar$ consisting of all those ultrafilters on $\omega$ which extend $F$. Then, the following are equivalent:
\begin{enumerate}
	\item $N_F$ contains a non-trivial convergent sequence,
	\item there is $X\in\cso$ such that $X\sub^*A$ for every $A\in F$,
	\item there is $X\in\cso$ such that $F\rstr X=Fr(X)$,
	\item the dual ideal $F^*$ is not tall,
	\item $\fF$ has non-empty interior in $\ostar$.\noproof
\end{enumerate}
\end{proposition}

%

We will now investigate the case when the space $S_F^*$ for a given filter $F$ contains non-trivial convergent sequences. Let us thus introduce a general scheme of constructing filters $F$ on $\omega$ such that the point $p_F$ is the limit of a convergent sequence in the space $S_F$.

Let $\seqn{A_n}$ be a sequence of pairwise disjoint (finite or infinite) non-empty subsets of $\omega$. For every $n\io$ let $F_n$ be a (not necessarily free) filter on $A_n$---we will say that the sequence \textit{$\seqn{F_n}$ is based on the sequence $\seqn{A_n}$}. Let us recall the following well-known definition, going back to Frol\'ik \cite{Fro67}:
\[\sum\big(\seqn{F_n}\big)=\Big\{A\in\cso\colon\ \big\{n\io\colon\ A\cap A_n\in F_n\big\}\in Fr\Big\}.\]
We will usually write simply $\sum F_n$ instead of $\sum\big(\seqn{F_n}\big)$. It is immediate that $\sum F_n$ is a free filter on $\omega$.

\begin{fact}
Let $\seqn{F_n}$ be a sequence of filters based on a sequence $\seqn{A_n}$. Fix $N\io$ and co-infinite $X\in\cso$. Then:
\begin{enumerate}
	\item $\bigcup_{n\ge N}A_n\in \sum F_n$,
	\item $\omega\sm A_N\in \sum F_n$, so $A_N\in\aA_{\sum F_n}$,
	\item $\sum\big(\seqn{F_n}\big)=\sum\big(\seqn{F_{n+N}}\big)$,
	\item $\sum\big(\seqn{F_n}\big)\subsetneq \sum\big(\seq{F_n}{n\in X}\big)$.
\noproof
\end{enumerate}
\end{fact}


The next lemma provides sufficient conditions for a filter $F$ so that $S_F$ contains a non-trivial convergent sequence.

\begin{lemma}\label{lemma:sf_f_convseq}
For every sequence $\seqn{F_n}$ based on a sequence $\seqn{A_n}$ and filter $F$ on $\omega$ such that $F\sub \sum F_n$, the point $p_F$ in the space $S_F$ is the limit of a non-trivial convergent sequence. More precisely:
\begin{enumerate}
	\item if there is a subsequence $\seqk{n_k}$ such that for every $k\io$ the filter $F_{n_k}$ is a principal filter on $A_{n_k}$, then there is a non-trivial sequence in $\omega$ convergent to $p_F$;
	\item if there is a subsequence $\seqk{n_k}$ such that for every $k\io$ the filter $F_{n_k}$ is a free filter on $A_{n_k}$ and $\omega\sm A_{n_k}\in F$, then there is a non-trivial sequence in $S_F^*$ convergent to $p_F$.
\end{enumerate}
\end{lemma}
\begin{proof}
(1) For every $k\io$ let $x_k\io$ be the only point of $\bigcap F_{n_k}$. Then, for every $A\in F$ there is $K\io$ such that $x_k\in A$ for every $k>K$. It follows by Lemma \ref{lemma:nf_convseq_frechet} that $\lim_{k\to\infty}x_k=p_F$.

(2) For every $k\io$ we have $\omega\sm A_{n_k}\in F$, so there is an ultrafilter $\uU_k$ on $\aA_F$ such that $A_{n_k}\in\uU_k$, $\uU_k\in S_F^*$, and $F_{n_k}\sub\uU_k$. It follows that $\lim_{k\to\infty}\uU_k=p_F$. Indeed, if $A\in F$, then $A\in \sum F_n$, so for almost all $k\io$ we have $A\cap A_{n_k}\in F_{n_k}$ and thus $A\in\uU_k$, and hence for almost all $k\io$ the point $\uU_k$ belongs to $\clopen{A}_F$, the clopen subset of $S_F$ induced by $A$. 
\end{proof}

In fact, in the proof of (2) we have shown more: if for each $k\io$ $X_k$ is the closed subset of $S_F^*$ consisting of all ultrafilters on $\aA_F$ extending $F_{n_k}$, then the sequence $\seqk{X_k}$ converges to $p_F$. 

The converse to Lemma \ref{lemma:sf_f_convseq} is also true.

\begin{lemma}\label{lemma:sf_convseq_f}
Let $F$ be a free filter on $\omega$ such that $S_F$ contains a non-trivial convergent sequence $\seqn{x_n}$. Then, there is a sequence $\seqn{F_n}$ of filters based on some sequence $\seqn{A_n}$ such that $F\sub \sum F_n$.
\end{lemma}
\begin{proof}
By Lemma \ref{lemma:sf_convseq_pf}, $\lim_{n\to\infty}x_n=p_F$, so we may assume that $x_n\neq p_F$ for every $n\io$. Using the Tietze Extension Theorem we may find $f\in C\big(S_F\big)$ such that $f\big(p_F\big)=0$ and $f\big(x_n\big)=1/(n+1)$ for every $n\io$, and hence we can easily construct a sequence $\seqn{A_n}$ of pairwise disjoint subsets of $\omega$ such that $A_n\in x_n$ and $A_n^c\in F$ for every $n\io$ ($A_n$'s may be finite or infinite). For every $n\io$ let $F_n=x_n\rstr A_n$. 
Then, $F\sub \sum F_n$. Indeed, let $A\in F$, then there is $N\io$ such that for every $n>N$ we have $A\in x_n$ and hence $A\cap A_n\in
F_n$. This yields that $A\in \sum F_n$.
\end{proof}

Putting Lemmas \ref{lemma:sf_f_convseq} and \ref{lemma:sf_convseq_f} together we obtain the following characterization of those filters $F$ for which the space $S_F$ contains a non-trivial convergent sequence (with the limit $p_F$, by Lemma \ref{lemma:sf_convseq_pf}).

\begin{theorem}
Let $F$ be a free filter on $\omega$. The space $S_F$ contains a non-trivial convergent sequence if and only if $F\sub \sum F_n$ for some sequence $\seqn{F_n}$ based on some $\seqn{A_n}$. \noproof
\end{theorem}

The following proposition provides a sufficient condition for a sequence $\seqn{F_n}$ so that the space $N_{\sum F_n}$ does not have any non-trivial convergent sequences (cf. condition (2) in Lemma \ref{lemma:sf_f_convseq}).

\begin{proposition}\label{prop:lf_no_frechet}
For every sequence $\seqn{F_n}$ of free filters based on a sequence $\seqn{A_n}$ of (necessarily infinite) subsets of $\omega$ and $F=\sum F_n$, there is no $X\in\cso$ such that $X\sub^*A$ for every $A\in F$. In particular, $S_F^*$ contains a non-trivial convergent sequence and $N_F$ does not.
\end{proposition}
\begin{proof}
Let $X\in\cso$. If $X\cap A_n\neq\emptyset$ for at most finitely many $n\io$, then $X\not\sub^*\bigcup_{n\ge N}A_n\in F$ for a sufficiently large number $N\io$. So let $M\in\cso$ be such that $X\cap A_n\neq\emptyset$ for all $n\in M$. For each $n\in M$ pick $x_n\in X\cap A_n$. Put $Y=\omega\sm\big\{x_n\colon n\in M\big\}$. Since for every $n\io$ we have $\omega\cap A_n=A_n\in F_n$ and the filter $F_n$ is free, we have also $Y\cap A_n\in F_n$ for every $n\io$, so $Y\in F$. But $x_n\in X\sm Y$ for every $n\in M$ and $M$ is infinite, so $X\not\sub^* Y$.

The second statement follows from Lemma \ref{lemma:sf_f_convseq}.(2) and Proposition \ref{prop:nf_convseq_char}.
\end{proof}

We will use Proposition \ref{prop:lf_no_frechet} to obtain a space $S_F$ with the JNP (induced by a convergent sequence in $S_F^*$) but such that $N_F$ does not have the BJNP---see Corollary \ref{cor:sf_jnp_nf_no_bjnp}.

We also have a counterpart for condition (1) of Lemma \ref{lemma:sf_f_convseq}.

\begin{proposition}\label{prop:lf_no_convseq_sf}
For every sequence $\seqn{F_n}$ of principal filters based on some sequence $\seqn{A_n}$ of (finite or infinite) subsets of $\omega$ and $F=\sum F_n$, there is $X\in\cso$ such that $F=Fr(X,\omega)$. 
In particular, $N_F$ contains a non-trivial convergent sequence, but there is no non-trivial convergent sequence in $S_F^*$.
\end{proposition}
\begin{proof}
For every $n\io$ let $z_n\io$ be the only point of $\bigcap F_n$---the set $X=\big\{z_n\colon\ n\io\big\}$ satisfies the condition that $X\sub^* A$ for every $A\in F$. Moreover, $X\in F$, since each $F_n$ is principal (so $\big\{z_n\big\}\in F_n$). Consequently, $F=Fr(X,\omega)$.

It follows by Proposition \ref{prop:nf_convseq_char} that $N_F$ contains a non-trivial sequence convergent to $p_F$. We show that there is no non-trivial convergent sequence in $S_F^*$. If $\omega\sub^*X$, then $F=Fr$ and $S_F^*=\big\{p_F\big\}$, so we are done. Assume then that $\omega\sm X\in\cso$ and, for the sake of contradiction, suppose there is a non-trivial sequence $\seqn{x_n\in S_F^*\sm\big\{p_F\big\}}$ such that $\lim_{n\to\infty}x_n=p_F$. 
Since $\seqn{x_n}$ converges to $p_F$ and $X\in F$, there is $N\io$ such that $X\in x_n$ for every $n>N$. But each $x_n$ is free, so $X\sm\{0,\ldots,k\}\in x_n$ for every $k\io$ and $n>N$. Since for every $A\in F$ we have $X\sub^* A$, for every $n>N$ and $A\in F$ we have $X\cap A\in x_n$, and hence $A\in x_n$, which implies that we cannot separate $p_F$ from any of the points $x_n$ with $n>N$---a contradiction, since $S_F$ is Hausdorff.
\end{proof}

\section{General characterizations of the BJNP and JNP of spaces $N_F$ and $S_F$\label{sec:nf_char_jnp_bjnp}}

In this section we present sufficient and necessary conditions for spaces $N_F$, $S_F$, and $S_F^*$ to have the BJNP or the JNP. 
We start the section by recalling some important facts which will be frequently used in what follows.

\begin{lemma}\label{lemma:basic_jnp}
Let $X$ be a space.
\begin{enumerate}
	\item If $X$ contains a non-trivial convergent sequence, then it has the JNP.
	\item If $X=\bo$ or $X=\os$, then $X$ does not have the BJNP nor the JNP.
	\item If $X$ has the JNP (resp. the BJNP), then $X$ has a JN-sequence (resp. a BJN-sequence) with disjoint supports.
	\item If $X$ is discrete, then $X$ does not have the BJNP.
\end{enumerate}
\end{lemma}
\begin{proof}
(1) is easy and well-known. For (2) see \cite{BKS1} or \cite{KSZgroth}. For (3) see \cite[Section 4]{MSZ} for the case of the JNP---the proof works without any changes for the BJNP, too. (4) easily follows from (3).
\end{proof}


Recall that a subset $A$ of a space $X$ is \textit{bounded}\footnote{Some authors use the name \textit{functionally bounded}.} in $X$ if for every $f\in C(X)$ we have $f\rstr A\in C^*(A)$.

\begin{lemma}\label{lemma:bounded_supports}
Let $X$ be a space and $\seqn{\mu_n}$ a sequence of finitely supported measures on $X$. Then, the following hold:
\begin{enumerate}
	\item if $\seqn{\mu_n}$ is a JN-sequence on $X$, then the union $\bigcup_{n\io}\supp\big(\mu_n\big)$ is bounded in $X$;
	\item if $\seqn{\mu_n}$ is a BJN-sequence on $X$ such that the union $\bigcup_{n\io}\supp\big(\mu_n\big)$ is contained in a subset $Y$ which is bounded in $X$, then it is a JN-sequence on $X$.
\end{enumerate}
\end{lemma}
\begin{proof}
For the proof of (1), see \cite[page 3024]{BKS1} or \cite[Proposition 4.1]{KMSZ}. 

We prove (2). Let $f\in C(X)$. There is $M>0$ such that $|f(x)|<M$ for every $x\in Y$. Let $r\colon\R\to[-M,M]$ be a retraction and put $g=r\circ f$. Since $g(x)=f(x)$ for every $x\in Y$ and $g\in C^*(X)$, we have  $\lim_{n\to\infty}\mu_n(f)=\lim_{n\to\infty}\mu_n(g)=0$, which proves that $\seqn{\mu_n}$ is a JN-sequence on $X$.
\end{proof}

Of course, the union of supports of measures from a (B)JN-sequence on a given space must necessarily be infinite. 

\subsection{Characterization in terms of probability measures on spaces $N_F$\label{sec:nf_bjnp_prob}}

We first characterize which sequences of measures on spaces $N_F$ are BJN-sequences or JN-sequences.

\begin{proposition}\label{prop:nf_bjnseq_char}
Let $F$ be a free filter on $\omega$ and $\seqn{\mu_n}$ a sequence of finitely supported measures on $N_F$. For each $n\io$ let $P_n=\big\{x\in\supp\big(\mu_n\big)\colon\ \mu_n(\{x\})>0\big\}$ and $N_n=\supp\big(\mu_n\big)\sm P_n$. Then, $\seqn{\mu_n}$ is a BJN-sequence on $N_F$ if and only if the following three conditions simultaneously hold:
\begin{enumerate}
	\item $\big\|\mu_n\big\|=1$ for every $n\io$,
	\item $\lim_{n\to\infty}\big\|\mu_n\rstr P_n\big\|=\lim_{n\to\infty}\big\|\mu_n\rstr N_n\big\|=1/2$,
	\item $\lim_{n\to\infty}\big\|\mu_n\rstr(\omega\sm A)\big\|=0$ for every $A\in F$.
\end{enumerate}
\end{proposition}
\begin{proof}
Assume that $\seqn{\mu_n}$ is a BJN-sequence on $N_F$. Then, (1) holds by the definition and (2) was essentially observed in \cite[Lemma 3.1]{MSZ} (basing on the fact that $\mu_n\big(N_F\big)\to 0$). 
If $A\in F$ 
and there is a subsequence $\seqk{\mu_{n_k}}$ such that $\big\|\mu_{n_k}\rstr(\omega\sm A)\big\|>0$ for every $k\io$ and $\lim_{k\to\infty}\big\|\mu_{n_k}\rstr(\omega\sm A)\big\|>0$, then the sequence
\[\seqk{\big(\mu_{n_k}\rstr(\omega\sm A)\big)/\big\|\mu_{n_k}\rstr(\omega\sm A)\big\|}\]
is a BJN-sequence on the discrete clopen subspace $\omega\sm A$ of $N_F$, contradicting Lemma \ref{lemma:basic_jnp}.(4). This proves (3).

Assume now that $\seqn{\mu_n}$ is a sequence of finitely supported measures on $N_F$ satisfying conditions (1)--(3). We will first show that $\mu_n(f)\to 0$ for every $f\in C^*\big(N_F\big)$ such that $f\big(p_F\big)=0$. Let thus $f\in C^*\big(N_F\big)$ be such a (non-zero) function and fix $\eps>0$. By the continuity of $f$, there is $A\in F$ such that $|f(n)|<\eps/2$ for every $n\in A$. By (3), there is $N\io$ such that for every $n>N$ we have $\big\|\mu_n\rstr(\omega\sm A)\big\|<\eps/(2\|f\|_\infty)$. It holds now:
\[\big|\mu_n(f)\big|\le\big|\big(\mu_n\rstr(\omega\sm A)\big)(f)\big|+\big|\big(\mu_n\rstr\ol{A}^{N_F}\big)(f)\big|\le\]
\[\|f\|_\infty\cdot\big\|\mu_n\rstr(\omega\sm A)\big\|+\big\|f\rstr\ol{A}^{N_F}\big\|_\infty\cdot\big\|\mu_n\rstr\ol{A}^{N_F}\big\|<\eps/2+\eps/2\cdot1=\eps\]
for every $n>N$, so $\lim_{n\to\infty}\mu_n(f)=0$.%

Let now $f\in C^*\big(N_F\big)$ be arbitrary. Notice that for every $n\io$ we have:
\[\mu_n\big(f-f\big(p_F\big)\cdot\chi_{N_F}\big)=\mu_n(f)-f\big(p_F\big)\cdot\mu_n\big(\chi_{N_F}\big)=\mu_n(f)-f\big(p_F\big)\Big(\mu_n\big(P_n\big)+\mu_n\big(N_n\big)\Big),\]
so, by (2) and the following equality (proved above) \[\lim_{n\to\infty}\mu_n\big(f-f\big(p_F\big)\cdot\chi_{N_F}\big)=0,\]
we get that $\lim_{n\to\infty}\mu_n(f)=0$. Consequently, $\seqn{\mu_n}$ is a BJN-sequence on $N_F$.
\end{proof}

\noindent Note that condition (3) in the above theorem may be equivalently stated as follows:

(3') $\lim_{n\to\infty}\big\|\mu_n\rstr\ol{A}^{N_F}\big\|=1$ for every $A\in F$.

\noindent Recall that for each $A\in F$ we have $\ol{A}^{N_F}=A\cup\big\{p_F\big\}$ and so $N_F=(\omega\sm A)\cup\ol{A}^{N_F}$.

To ensure that a given sequence of measures on a space $N_F$ is a JN-sequence, we need to add one supplementary condition to the conclusion of Proposition \ref{prop:nf_bjnseq_char}.

\begin{proposition}\label{prop:nf_jnseq_char}
Let $F$ be a free filter on $\omega$ and $\seqn{\mu_n}$ a sequence of finitely supported measures on $N_F$. Then, $\seqn{\mu_n}$ is a JN-sequence on $N_F$ if and only if it satisfies conditions (1)--(3) of Proposition \ref{prop:nf_bjnseq_char} and, in addition, the following one:
\begin{enumerate}
\setcounter{enumi}{3}
	\item $\bigcup_{n\io}\supp\big(\mu_n\big)\sub^* A$ for every $A\in F$.
\end{enumerate}
\end{proposition}
\begin{proof}
Assume that $\seqn{\mu_n}$ satisfies conditions (1)--(3) of Proposition \ref{prop:nf_bjnseq_char} and additionally condition (4). By Proposition \ref{prop:nf_bjnseq_char}, $\seqn{\mu_n}$ is a BJN-sequence on $N_F$. Put 
\[Y=\bigcup_{n\io}\supp\big(\mu_n\big)\sm\big\{p_F\big\}.\]
It follows, by (4), that $Y\sub^* A$ for every $A\in F$, which means, by Lemmas \ref{lemma:nf_convseq_frechet} and \ref{lemma:sf_convseq_pf}, that $\seq{n}{n\in Y}$ is a sequence convergent to $p_F$. Put $X=Y\cup\big\{p_F\big\}$. 
Since $X$ with the inherited topology is compact, it is bounded in $N_F$ and hence, by Lemma \ref{lemma:bounded_supports}.(2), $\seqn{\mu_n}$ is a JN-sequence on $N_F$.

Conversely, assume that $\seqn{\mu_n}$ is a JN-sequence on $N_F$. Since it is a BJN-sequence on $N_F$, by Proposition \ref{prop:nf_bjnseq_char} we need only to prove condition (4). So, for the sake of contradiction, suppose there is $A\in F$ such that the set
\[X=\bigcup_{n\io}\supp\big(\mu_n\big)\cap(\omega\sm A)\]
is infinite. We will construct an unbounded function $f\in C\big(N_F\big)$ such that $f\rstr A\equiv 0$ and $\limsup_{n\to\infty}\big|\mu_n\big(f\rstr(\omega\sm A)\big)\big|>0$. Since $X$ is infinite and each $\supp\big(\mu_n\big)$ is finite, we can find sequences $\seqk{x_k\in X}$ and $\seqk{n_k\io}$ such that
\[x_k\in\supp\big(\mu_{n_k}\big)\sm\bigcup_{i=0}^{k-1}\supp\big(\mu_{n_i}\big)\]
for every $k\io$. We first define $f$ inductively on $\big\{x_k\colon k\io\big\}$. For $k=0$ let $f\big(x_0\big)=0$ 
and for $k>0$ define $f\big(x_k\big)$ as follows:
\[f\big(x_k\big)=\Big(1-\sum_{i=0}^{k-1}f\big(x_i\big)\cdot\mu_{n_k}\big(\big\{x_i\big\}\big)\Big)\cdot\mu_{n_k}\big(\big\{x_k\big\}\big)^{-1}.\]
For $x\in X\sm\big\{x_k\colon\ k\io\big\}$ or $x\in\ol{A}^{N_F}$, let again $f(x)=0$, so $f\rstr\ol{A}^{N_F}\equiv 0$. 
Since $A\in F$ and thus the subspace $N_F\sm\ol{A}^{N_F}=\omega\sm A$ of $N_F$ is clopen and discrete, $f$ is continuous and so $f\in C\big(N_F\big)$. Then, as for every $k>0$ we have:
\[\supp\big(\mu_{n_k}\big)\cap\big\{x_i\colon\ i\io\big\}\sub\big\{x_0,\ldots,x_k\big\},\]
it also holds:
\[\mu_{n_k}(f)=\sum_{i=0}^{k-1}f\big(x_i\big)\cdot\mu_{n_k}\big(\big\{x_i\big\}\big)+f\big(x_k\big)\cdot\mu_{n_k}\big(\big\{x_k\big\}\big)=\]
\[\sum_{i=0}^{k-1}f\big(x_i\big)\cdot\mu_{n_k}\big(\big\{x_i\big\}\big)+\Big(1-\sum_{i=0}^{k-1}f\big(x_i\big)\cdot\mu_{n_k}\big(\big\{x_i\big\}\big)\Big)\cdot\mu_{n_k}\big(\big\{x_k\big\}\big)^{-1}\cdot\mu_{n_k}\big\{\big\{x_k\big\}\big)=1,\]
which implies that $\limsup_{n\to\infty}\mu_n(f)\ge1$, contradicting the fact that $\seqn{\mu_n}$ is a JN-sequence on $N_F$. 

It follows that $\bigcup_{n\io}\supp\big(\mu_n\big)\sub^*A$ for every $A\in F$ and hence (4) holds.
\end{proof}

\begin{theorem}
\label{theorem:nf_jnp}
For every free filter $F$ on $\omega$, the space $N_F$ has the JNP if and only if $N_F$ contains a non-trivial convergent sequence.
\end{theorem}
\begin{proof}
If $\seqn{\mu_n}$ is a JN-sequence on $N_F$, then, by Proposition \ref{prop:nf_jnseq_char}.(4),
\[\bigcup_{n\io}\supp\big(\mu_n\big)\sub^*A\]
for every $A\in F$, which, by Proposition \ref{prop:nf_convseq_char}, means that $N_F$ contains a non-trivial convergent sequence. The converse is well-known.
\end{proof}

Recall that in Proposition \ref{prop:nf_convseq_char} we have provided several equivalent conditions for a space $N_F$ to contain a non-trivial convergent sequence.

From Proposition \ref{prop:nf_bjnseq_char} we can the derive the proof of Theorem \ref{theorem:main_theorem_BJNP_prob} from Introduction, which can be thought of as a BJNP-analogon of Theorem \ref{theorem:nf_jnp}, since any non-trivial convergent sequence $\seqn{x_n\io}$ in a given space $N_F$ induces the sequence $\seqn{\delta_{x_n}}$ of trivial finitely supported probability measures satisfying conditions (1) and (2) of Theorem \ref{theorem:main_theorem_BJNP_prob}.

\begin{proof}[\textbf{Proof of Theorem \ref{theorem:main_theorem_BJNP_prob}}]
Assume that $N_F$ has the BJNP. By Lemma \ref{lemma:basic_jnp}.(3), there is a disjointly supported BJN-sequence $\seqn{\theta_n}$ on $N_F$, so we may assume that $p_F\not\in\supp\big(\theta_n\big)$ for every $n\io$. For each $n\io$ put $\mu_n=\big|\theta_n\big|$, so that $\big\|\mu_n\big\|=1$, $\supp\big(\mu_n\big)\sub\omega$, and $\mu_n\ge0$, and hence $\mu_n$ is a probability measure on $N_F$. By Proposition \ref{prop:nf_bjnseq_char}.(3) (cf. also condition (3')) we have:
\[\lim_{n\to\infty}\mu_n(A)=\lim_{n\to\infty}\big\|\theta_n\rstr A\big\|=\lim_{n\to\infty}\big\|\theta_n\rstr\ol{A}^{N_F}\big\|=1\]
for every $A\in F$. Thus, conditions (1) and (2) are satisfied.

Assume now that there is a sequence $\seqn{\mu_n}$ of finitely supported probability measures on $N_F$ satisfying conditions (1) and (2). For each $n\io$ put:
\[\nu_n=\frac{1}{2}\big(\delta_{p_F}-\mu_n\big).\]
By Proposition \ref{prop:nf_bjnseq_char} it is immediate that $\seqn{\nu_n}$ is a BJN-sequence on $N_F$.
\end{proof}


The above proof also implies that condition (2) can be rephrased in the following alternative way:

(2') $\mu_n(f)\to\delta_{p_F}(f)$ for every $f\in C_p^*(X)$.

\noindent (Recall that $\delta_{p_F}(f)=f\big(p_F\big)$.)
%

\medskip

Theorem \ref{theorem:main_theorem_BJNP_prob} may be used to obtain a characterization of filters $F$ such that $N_F$ has the BJNP in terms of single probability measures on $\omega$ and closed subsets of $\ostar$. Note that a closed subset $\fF$ of $\ostar$, having similar properties to those described in Propositions \ref{prop:one_measure_bjnp} and \ref{prop:bjnp_one_measure} and Remark \ref{rem:one_measure_bjnp_constr}, was used in \cite[Section 5]{KSZgroth} to construct a compact space $K$ such that its Banach function space $C(K)$ has the $\ell_1$-Grothendieck property but does not have the Grothendieck property.

\begin{proposition}\label{prop:one_measure_bjnp}
Let $\mu$ be a probability measure on $\bo$ with $\mu(\omega)=1$ and let $\seqn{A_n}$ be a sequence of pairwise disjoint subsets of $\omega$ such that $\mu\big(A_n\big)>0$ for every $n\io$. Assume that $\fF$ is a closed subset of $\ostar$ having the following property 
\textup{($\dagger$)}:
\begin{equation}
\tag{$\dagger$}\begin{aligned}\text{for every clopen subset $U$ of $\bo$ containing $\fF$ it holds }
\lim_{n\to\infty}\frac{\mu\big(A_n\cap U\big)}{\mu\big(A_n\big)}=1.\end{aligned}
\end{equation}
Let $F$ be the free filter on $\omega$ defined by $F=\bigcap\fF$. Then, $N_F$ has the BJNP.
\end{proposition}
\begin{proof}
For each $n\io$ let $B_n$ be a finite subset of $A_n$ such that
\[\mu\big(A_n\sm B_n\big)/\mu\big(A_n\big)<1/2^{n+1};\]
note that $\mu\big(B_n\big)>0$. Then, for each $n\io$ and $A\in\wo$, set:
\[\mu_n(A)=\mu\big(B_n\cap A\big)/\mu\big(B_n\big).\]
It follows that each $\mu_n$ is a finitely supported probability measure with $\supp\big(\mu_n\big)\sub B_n\sub\omega$. 

Fix $A\in F$. Then, for the corresponding clopen subset of $\bo$, it holds $\fF\sub\clopen{A}_\omega$. We trivially have:
\[\tag{$*$}1\ge\mu_n(A)\ge\mu\big(B_n\cap A\big)/\mu\big(A_n\big),\]
and
\[\mu\big(A_n\cap A\big)-\mu\big(B_n\cap A\big)=\mu\big(\big(A_n\sm B_n\big)\cap A\big)\le\mu\big(A_n\sm B_n\big)<\mu\big(A_n\big)/2^{n+1},\]
so, by ($\dagger$),
\[\lim_{n\to\infty}\mu\big(B_n\cap A\big)/\mu\big(A_n\big)=1,\]
and hence, by ($*$), $\lim_{n\to\infty}\mu_n(A)=1$, too. By Theorem \ref{theorem:main_theorem_BJNP_prob} the proof is finished.
\end{proof}

\begin{remark}\label{rem:one_measure_bjnp_constr}
Let $\mu$ and $\seqn{A_n}$ be as in Proposition \ref{prop:one_measure_bjnp}. Set:
\[\mathfrak{F}=\big\{\fF\colon\ \fF\sub\ostar\text{ is a closed subset having property }(\dagger)\big\}.\]
If $\fF,\fF'$ are closed subsets of $\ostar$ such that $\fF\sub\fF'$ and $\fF\in\mathfrak{F}$, then $\fF'\in\mathfrak{F}$. Also, for every non-empty $\mathfrak{G}\sub\mathfrak{F}$ its intersection $\bigcap\mathfrak{G}$ has property ($\dagger$), i.e.  $\bigcap\mathfrak{G}\in\mathfrak{F}$. (To see this, first prove it for finite $\mathfrak{G}$'s by induction on $|\mathfrak{G}|$, then assume that $\bigcap\mathfrak{G}\not\in\mathfrak{F}$ and use the compactness argument). So, $\bigcap\mathfrak{F}\in\mathfrak{F}$. Consequently, $\mathfrak{F}$ is a principal filter in the family of all closed subsets of $\ostar$ (ordered by inclusion).

Let $\gG$ be the closed subset of $\ostar$ defined by the following equivalence:
\[x\in\gG\quad\Longleftrightarrow\quad\forall\text{ clopen neighborhood }U\text{ of }x\text{ in }\bo\colon\ \limsup_{n\to\infty}\frac{\mu\big(A_n\cap U\big)}{\mu\big(A_n\big)}>0.\]
Then, the following equality holds:
\[\gG=\bigcap\mathfrak{F}.\]
In particular, $\gG$ has property ($\dagger$), too.
\end{remark}

The converse to Proposition \ref{prop:one_measure_bjnp} also holds.

\begin{proposition}\label{prop:bjnp_one_measure}
Let $F$ be a free filter on $\omega$ such that the space $N_F$ has the BJNP. Then, there exist a probability measure $\mu$ on $\bo$ with $\mu(\omega)=1$ and a sequence $\seqn{A_n}$ of pairwise disjoint subsets of $\omega$ such that $\mu\big(A_n\big)>0$ for every $n\io$ and
\[\lim_{n\to\infty}\frac{\mu\big(A_n\cap U\big)}{\mu\big(A_n\big)}=1\]
for every clopen subset $U$ of $\bo$ containing the closed subset $\fF$ of $\ostar$ which consists of all ultrafilters extending $F$.
\end{proposition}
\begin{proof}
Let $\seqn{\mu_n}$ be a sequence of disjointly supported probability measures on $N_F$ from Theorem \ref{theorem:main_theorem_BJNP_prob}. Set:
\[\mu=\sum_{n\io}\mu_n/2^{n+1}.\]
(Here, we treat every $\mu_n$ as a measure on $\bo$.) Then, $\mu$ is a probability measure on $\bo$ such that $\mu(\omega)=1$. For each $n\io$, let $A_n=\supp\big(\mu_n\big)$, so $\mu\big(A_n\big)=1/2^{n+1}>0$.

Let $U$ be a clopen subset of $\bo$ containing $\fF$. There is $A\in\wo$ such that $U=\clopen{A}_\omega$. Consequently, $A\in F$, and so
\[\lim_{n\to\infty}\frac{\mu\big(A_n\cap U\big)}{\mu\big(A_n\big)}=\lim_{n\to\infty}\mu_n(A)=1,\]
which finishes the proof.
\end{proof}


\subsection{Characterization in terms of probability measures on spaces $S_F$ and $S_F^*$}

In this short subsection we present analogons of results from the previous subsection for spaces $S_F$ and $S_F^*$. Recall that, for every set $A$ belonging to a free filter $F$ on $\omega$, the clopen subset $\clopen{\omega\sm A}_F$ (resp. $\clopen{\omega\sm A}_F^*$) of $S_F$ (resp. of $S_F^*$) is either finite (resp. empty) or homeomorphic to $\bo$ (resp. to $\os$), hence it does not have the JNP (by Lemma \ref{lemma:basic_jnp}.(2)). The proof of the next theorem is thus basically identical to the one of Proposition \ref{prop:nf_bjnseq_char}, so we omit it.


\begin{theorem}\label{prop:sf_jnseq_char}
Let $F$ be a free filter on $\omega$ and let $X=S_F$ or $X=S_F^*$. Let $\seqn{\mu_n}$ be a sequence of finitely supported measures on $X$. For each $n\io$ let $P_n=\big\{x\in\supp\big(\mu_n\big)\colon\ \mu_n(\{x\})>0\big\}$ and $N_n=\supp\big(\mu_n\big)\sm P_n$. Then, $\seqn{\mu_n}$ is a JN-sequence on $X$ if and only if the following three conditions simultaneously hold:
\begin{enumerate}
	\item $\big\|\mu_n\big\|=1$ for every $n\io$,
	\item $\lim_{n\to\infty}\big\|\mu_n\rstr P_n\big\|=\lim_{n\to\infty}\big\|\mu_n\rstr N_n\big\|=1/2$,
	\item if $X=S_F$, then $\lim_{n\to\infty}\big\|\mu_n\rstr[\omega\sm A]_F\big\|=0$ for every $A\in F$, or,\\ if $X=S_F^*$, then $\lim_{n\to\infty}\big\|\mu_n\rstr[\omega\sm A]_F^*\big\|=0$ for every $A\in F$.\noproof
\end{enumerate}
\end{theorem}

The following proposition is a counterpart of Theorem \ref{theorem:nf_jnp} for spaces $S_F$ and $S_F^*$. Note that in Theorem \ref{theorem:nf_jnp} we did not put any restrictions on the sizes of supports.

\begin{proposition}\label{prop:sf_convseq_supps2}
Let $F$ be a free filter on $\omega$ and let $X=S_F$ or $X=S_F^*$. Then, the space $X$ contains a non-trivial sequence convergent to $p_F$ if and only if there is a JN-sequence $\seqn{\mu_n}$ on $X$ such that $\big|\supp\big(\mu_n\big)\big|=2$ for every $n\io$.
\end{proposition}

In the proof we will use the following elementary observation.

\begin{lemma}\label{lemma: general_seq_of_sets}
Let $\seqn{A_n}$ be a sequence of non-empty subsets of a space $X$ of sizes bounded by some constant $M\in\omega$.
Then, for any point $x\in X$, at least one of the following conditions holds:
\begin{enumerate}
    \item there exist an increasing sequence $\seqk{n_k\io}$ and points $x_k\in A_{n_k}$, $k\in \omega$, such that $\lim_{k\to\infty}x_k = x$;
    \item there exist $S\in\cso$ and an open neighborhood $U$ of $x$ such that $U\cap A_n = \emptyset$ for all $n\in S$.
\end{enumerate}
\end{lemma}

\begin{proof}
Assuming the negation of condition (1), one can easily construct sets $S$ and $U$ required in (2) by induction on M.
\end{proof} 

\begin{proof}[Proof of Proposition \ref{prop:sf_convseq_supps2}]
We will prove the proposition for $X=S_F$ only as the proof for $X=S_F^*$ is similar.

The implication in the right direction follows from Lemma \ref{lemma:basic_jnp}.(1), so assume that $S_F$ admits a JN-sequence $\seqn{\mu_n}$ such that all supports $A_n = \supp\big(\mu_n\big)$ have size $2$. By (the proof of) \cite[Lemma 6.6]{MSZ}, we may assume that the sets $A_n$ are pairwise disjoint. 

Suppose that, for some $A\in F$ and $S\in\cso$, we have $A_n\cap\clopen{A}_F = \emptyset$ for every $n\in S$. Then, the clopen $\clopen{A^c}_F$ admits a JN-sequence $\seq{\mu_n}{n\in S}$,  despite that it is homeomorphic to $\bo$, which is a contradiction. Hence, Lemma \ref{lemma: general_seq_of_sets}, applied for the sequence $\seqn{A_n}$ and the point $x = p_F$, gives a sequence of pairwise distinct points $\seqk{x_k}$ convergent to $p_F$.
\end{proof}

By \cite[Theorem 6.12]{MSZ}, in Proposition \ref{prop:sf_convseq_supps2} we may exchange the condition that $\big|\supp\big(\mu_n\big)\big|=2$ for every $n\io$ for the condition that there exists $M\ge2$ such that $\big|\supp\big(\mu_n\big)\big|\le M$ for every $n\io$. In \cite[Section 4]{BKS1}, it is proved that the space $S_{F_d}$, where $F_d=\zZ^*$ (see Example \ref{example:idealZ}), has the JNP but does not contain any non-trivial convergent sequences---it follows that every JN-sequence on $S_{F_d}$ consists of measures with supports having sizes not bounded by any constant $M$ (cf. \cite[Proposition 6.2]{MSZ}). Notice that in order to prove that $S_{F_d}$ has the JNP, a sequence of the form $\frac{1}{2}\big(\delta_{p_F}-\mu_n\big)$, where each $\mu_n$ is a probability measure, was used.

Applying Theorem \ref{prop:sf_jnseq_char}, we obtain the following result---we leave the proof to the reader as it is almost identical to the one of Theorem \ref{theorem:main_theorem_BJNP_prob}. 


\begin{corollary}\label{cor:sf_jnp_char}
Let $F$ be a free filter on $\omega$ and let $X=S_F$ or $X=S_F^*$. Then, $X$ has the JNP if and only if there is a (disjointly supported) sequence $\seqn{\mu_n}$ of finitely supported probability measures on $X$ such that:
\begin{enumerate}
	\item 
$p_F\not\in\supp\big(\mu_n\big)$ for every $n\io$,
	\item if $X=S_F$, then $\lim_{n\to\infty}\mu_n\big(\clopen{A}_F\big)=1$ for every $A\in F$, or,\\ if $X=S_F^*$, then $\lim_{n\to\infty}\mu_n\big(\clopen{A}_F^*\big)=1$ for every $A\in F$.\noproof
\end{enumerate}
\end{corollary}
%


\subsection{The BJNP of spaces $N_F$ vs. the JNP of spaces $S_F$}

We will now briefly study relations between spaces $S_F$ and $N_F$ in the context of the BJNP and JNP.

The compactness of spaces $S_F$ immediately implies the following proposition.

\begin{proposition}\label{prop:nf_bjnp_sf_jnp}
For every free filter $F$ on $\omega$, if there is a BJN-sequence of measures on the space $N_F$, then the same sequence is a JN-sequence on the space $S_F$. In particular, if $N_F$ has the BJNP, then $S_F$ has the JNP.
\noproof
\end{proposition}


\begin{theorem}\label{cor:sf_jnp_nf_sfs}
Let $F$ be a free filter on $\omega$. Then, the space $S_F$ has the JNP if and only if $N_F$ has the BJNP or $S_F^*$ has the JNP.
\end{theorem}
\begin{proof}
If $S_F^*$ has the JNP, then obviously $S_F$ has it, too. Similarly, if $N_F$ has the BJNP, then by Proposition \ref{prop:nf_bjnp_sf_jnp} $S_F$ has the JNP.

Assume now that $S_F$ has the JNP. Let $\seqn{\mu_n}$ be a sequence of probability measures like in Corollary \ref{cor:sf_jnp_char}. 
We consider two cases.

1) There exists a subsequence $\seqk{\mu_{n_k}}$ such that $\mu_{n_k}(\omega) \ge 1/2$ for every $k\io$. If for every $k\io$ we put:
\[\nu_k=\big(\mu_{n_k}\rstr\omega\big)\big/\mu_{n_k}(\omega),\]
then $\nu_k$ is a finitely supported probability measure on $N_F$ such that $\supp\big(\nu_k\big)\sub\omega$, as well as for every $A\in F$ we have:
\[\nu_k(\omega\sm A)=\mu_{n_k}(\omega\sm A)\big/\mu_{n_k}(\omega)\le 2\mu_{n_k}\big(\clopen{\omega\sm A}_F\big),\]
which, by condition (2) of Corollary \ref{cor:sf_jnp_char}, tends to $0$ as $k\to\infty$. By Theorem \ref{theorem:main_theorem_BJNP_prob}, $N_F$ has the BJNP.

2) There exists $N\io$ such that for every $n>N$ we have $\mu_n(\omega)<1/2$ and hence $\mu_n\big(S_F^*\big)>1/2$. For every $n>N$ put:
\[\nu_n=\big(\mu_n\rstr S_F^*\big)\big/\mu_n\big(S_F^*\big),\]
so $\nu_n$ is a finitely supported probability measure on $S_F^*$ such that $p_F\not\in\supp\big(\nu_n\big)$ (by condition (1) of Corollary \ref{cor:sf_jnp_char}). For every $A\in F$ and $n>N$ we have:
\[\nu_n\big(\clopen{\omega\sm A}_F^*\big)=\mu_n\big(\clopen{\omega\sm A}_F\sm\omega\big)\big/\mu_n\big(S_F^*\big)\le2\mu_n\big(\clopen{\omega\sm A}_F\big),\]
which, again by condition (2) of Corollary \ref{cor:sf_jnp_char}, converges to $0$ as $n\to\infty$. By Corollary \ref{cor:sf_jnp_char}, $S_F^*$ has the JNP.
\end{proof}

A converse to Proposition \ref{prop:nf_bjnp_sf_jnp} may not hold. Corollary \ref{cor:sf_jnp_nf_no_bjnp} shows that $S_F$ may have the JNP induced by a convergent sequence contained in $S_F^*$, but $N_F$ fails to have the BJNP. However, by the next proposition, if $S_F$ has the JNP witnessed by a JN-sequence with supports contained in $\omega$, then $N_F$ has the BJNP. 



\begin{proposition}\label{prop:sf_jnp_nf_bjnp}
Let $F$ be a free filter on $\omega$ and $\seqn{\mu_n}$ a sequence of measures on the space $S_F$ such that $\supp\big(\mu_n\big)\sub N_F$ for every $n\io$. If $\seqn{\mu_n}$ is a JN-sequence on $S_F$, then it is a BJN-sequence on $N_F$.
\end{proposition}
\begin{proof}
Assume that $\seqn{\mu_n}$ is a JN-sequence on $S_F$ and let $f\in C^*\big(N_F\big)$. By Lemma \ref{lemma:nf_cstar_embedded_sf}.(1), there is $f'\in C^*\big(S_F\big)=C\big(S_F\big)$ such that $f=f'\rstr N_F$. Since $\supp\big(\mu_n\big)\sub N_F$ for every $n\io$, it holds $\mu_n(f)=\mu_n\big(f'\rstr N_F\big)=\mu_n(f')\to0$ as $n\to\infty$. It follows that $\lim_{n\to\infty}\mu_n(f)=0$ for every $f\in C^*\big(N_F\big)$, so $N_F$ has the BJNP.
\end{proof}

The following result is analogous to Proposition \ref{prop:sf_convseq_supps2} and proved in a very similar way.

\begin{proposition}\label{prop:nf_bjnp_supp_m}
Let $F$ be a free filter on $\omega$. Then, the space $N_F$ contains a non-trivial convergent sequence if and only if there are a BJN-sequence $\seqn{\mu_n}$ on $N_F$ and an integer $M\ge 2$ such that $\big|\supp\big(\mu_n\big)\big|\le M$ for every $n\io$.
\end{proposition}
\begin{proof}
The implication in the right direction is clear (cf. Lemma \ref{lemma:basic_jnp}.(1)). To see the converse, let $\seqn{\mu_n}$ be a BJN-sequence on $N_F$ such that there is a common bound for the sizes of all supports $A_n = \supp\big(\mu_n\big)$. By (the proof of) \cite[Theorem 4.3]{MSZ}, we may assume that the sets $A_n$ are pairwise disjoint. 

Suppose that, for some $A\in F$ and $S\in\cso$, we have $A_n\cap A = \emptyset$ for each $n\in S$. Then, the discrete clopen subspace $\omega\sm A$ of $N_F$  admits a BJN-sequence $\seq{\mu_n}{n\in S}$, which is a contradiction with Lemma \ref{lemma:basic_jnp}.(4). Hence, Lemma \ref{lemma: general_seq_of_sets}, applied for the sequence $\seqn{A_n}$ and the point $x = p_F$, gives a sequence of pairwise distinct points $\seqk{x_k}$ convergent to $p_F$.
\end{proof}

\subsection{Free sums of filters and the (B)JNP of spaces $N_F$\label{sec:free_sums}}

In this subsection we will make some observations concerning free sums of filters.
Recall that, given free filters $F_0$ and $F_1$ on $\omega$, their \textit{free sum} $F_0\oplus F_1$ is a free filter on $\omega\times\{0,1\}$ defined by
\[ F_0\oplus F_1 = \big\{(A_0\times\{0\})\cup (A_1\times\{1\})\colon\ A_0\in F_0,\ A_1\in F_1\big\}.\]
By mapping $\omega\times\{0,1\}$ onto $\omega$ via a bijection, we can assume that $F_0\oplus F_1$ is actually a filter on $\omega$. Using this identification, we can unambiguously abuse the notation and speak about the space $N_{F_0\oplus F_1}$. Observe then that $N_{F_0\oplus F_1}$ is the union of two closed subspaces $N'_{F_i} = (\omega\times\{i\})\cup\big\{p_{F_0\oplus F_1}\big\}$, $i=0,1$, such that $N'_{F_0}\cap N'_{F_1}=\big\{p_{F_0\oplus F_1}\big\}$, and which can be identified with the spaces $N_{F_i}$, $i=0,1$. For each $i=0,1$, the map $r_i\colon N_{F_0\oplus F_1}\to N'_{F_i}$ sending the subset $N'_{F_{1-i}}$ to the point $p_{F_0\oplus F_1}$ and being the identity on $N'_{F_i}$ is a continuous retraction of $N_{F_0\oplus F_1}$ onto  $N'_{F_i}$. Therefore the subspaces $N'_{F_0}$ and $N'_{F_1}$ are $C$-embedded in $N_{F_0\oplus F_1}$, and we can extend any (bounded) continuous function on each $N'_{F_i}$ to a (bounded) continuous function on $N_{F_0\oplus F_1}$. 

\begin{proposition}\label{prop:product_filters}
Let $F_0$ and $F_1$ on $\omega$ be free filters on $\omega$. Then, the space $N_{F_0\oplus F_1}$ has the JNP (BJNP) if and only if $N_{F_0}$ has the JNP (BJNP) or $N_{F_1}$ has the JNP (BJNP).
\end{proposition}
\begin{proof}
The implication in the left direction follows immediately from the obvious fact that if a subspace $Y$ of a space $X$ has the  JNP (BJNP), then $X$ also has this property.

To prove the implication in the right direction, let us assume that $N_{F_0\oplus F_1}$ has the BJNP and take a sequence $\seqn{\mu_n}$ of finitely supported probability measures like in Theorem \ref{theorem:main_theorem_BJNP_prob}. Note that $p_{F_0\oplus F_1}\not\in\supp\big(\mu_n\big)$ for every $n\io$. There is a subsequence $\seqk{\mu_{n_k}}$ such that the limit $\lim_{k\to\infty}\mu_{n_k}\big(N'_{F_0}\big)$ exists---denote its value by $\alpha$. We have $\lim_{k\to\infty}\mu_{n_k}\big(N'_{F_1}\big)=1-\alpha$. If $\alpha>0$, let $i=0$ and $\alpha_0 = \alpha$, otherwise set $i=1$ and $\alpha_1 = 1 - \alpha=1$. By omitting several first elements of the sequence, we may assume that $\mu_{n_k}\big(N'_{F_i}\big)>0$ for every $k\io$.

We claim that for every $A\in F_0\oplus F_1$ we have $\lim_{k\to\infty}\mu_{n_k}\big(N'_{F_i}\cap A\big)=\alpha_i$. If not, then there are $A\in F_0\oplus F_1$, a subsequence $\seql{\mu_{n_{k_l}}}$, and $\beta\in[0,\alpha_i)$, such that $\lim_{l\to\infty}\mu_{n_{k_l}}\big(N'_{F_i}\cap A\big)=\beta$. But then \[\lim_{l\to\infty}\mu_{n_{k_l}}\big(N'_{F_{1-i}}\cap A\big)=1-\beta>1-\alpha_i= \lim_{l\to\infty}\mu_{n_{k_l}}\big(N'_{F_{1-i}}\big),\]
which is impossible.

For each $k\io$ let $\nu_k=\big(\mu_{n_k}\rstr N'_{F_i}\big)\big/\mu_{n_k}\big(N'_{F_i}\big)$. It follows that each $\nu_k$ is a finitely supported probability measure on $N'_{F_i}$ with $\supp\big(\nu_k\big)\sub\omega\cap N'_{F_i}$. Identifying $N_{F_i}$ with $N'_{F_i}$, for every $A\in F_i$ we have:
\[\lim_{k\to\infty}\nu_k(A)=\lim_{k\to\infty}\big(\mu_{n_k}\rstr N'_{F_i}\big)(A)\big/\mu_{n_k}\big(N'_{F_i}\big)=\lim_{k\to\infty}\mu_{n_k}\big( N'_{F_i}\cap A\big)\big/\mu_{n_k}\big(N'_{F_i}\big)=1,\]
so, by Theorem \ref{theorem:main_theorem_BJNP_prob}, the space $N_{F_i}$ has the BJNP.

The case of the JNP is similar but, thanks to Theorem \ref{theorem:nf_jnp}, much easier.
\end{proof}

The proof of the following simple lemma is left to the reader.

\begin{lemma}\label{lemma:fr_f1_f2_iso}
Let $F_1$ and $F_2$ be free filters on $\omega$ such that $F_i\rstr A\neq Fr(A)$ for any $A\in\cso$ and $i=1,2$. Then, the filters $F_1\oplus Fr$ and $F_2\oplus Fr$ are isomorphic if and only if $F_1$ and $F_2$ are isomorphic.

\noproof
\end{lemma}

\subsection{Kat\v{e}tov reductions and the (B)JNP of spaces $N_F$ and $S_F$}

We finish Section \ref{sec:nf_char_jnp_bjnp} investigating the issue of transferring the BJNP or JNP from one space $N_F$ onto another via a Kat\v{e}tov reduction. Let us recall here that the inclusion $F\subseteq G$, for filters on $\omega$, implies that $F\le_{K}G$.

\begin{proposition}\label{prop:rk_bjnp}
Let $F$ and $G$ be free filters on $\omega$ such that $F\le_{K}G$. Then,
\begin{enumerate}
	\item if $N_G$ has the JNP, then $N_F$ has the JNP;
	\item if $N_G$ has the BJNP, then $N_F$ has the BJNP;
	\item if $S_G$ has the JNP, then $S_F$ has the JNP.
\end{enumerate}
\end{proposition}
\begin{proof}

(1) Let $f\colon\omega\to\omega$ be a witness for $F\le_K G$. Assume that $N_G$ has the JNP. By Theorem \ref{theorem:nf_jnp} and Proposition \ref{prop:nf_convseq_char}, there is $X\in\cso$ such that $X\sub^*A$ for every $A\in G$. Let $Y=f[X]$. It follows that $Y\sub^*A$ for every $A\in F$. Indeed, if there is $A\in F$ such that $Y\sm A$ is infinite, then $X\sm f^{-1}[A]$ is infinite, a contradiction, since $f^{-1}[A]\in G$ and so $X\sub^* f^{-1}[A]$. One can show in a similar way that $Y$ is infinite. Thus, $N_F$ has the JNP, too.

(2) If $G=Fr$, then $N_G$ has the JNP, so by (1) $N_F$ has the JNP, and in particular the BJNP, too. So assume that $G\neq Fr$ and that $N_G$ has the BJNP. Let $f\colon\omega\to\omega$ be such a function that $F\sub f(G)$. Let $\seqn{\mu_n}$ be a sequence of finitely supported probability measures on $N_G$ as in Theorem \ref{theorem:main_theorem_BJNP_prob}. For each $n\io$ put:
\[\mu_n'=\sum_{x\in\supp(\mu_n)}\mu_n(\{x\})\cdot\delta_{f(x)};\]
then, $\big\|\mu_n'\big\|=1$, $\supp\big(\mu_n'\big)\sub\omega$ and $\mu_n'\ge0$. We also have that $\lim_{n\to\infty}\mu_n'(A)=1$ for every $A\in F$. Indeed, if there was $A\in F$ such that $\limsup_{n\to\infty}\mu_n'(\omega\sm A)>0$, then $f^{-1}[A]\in G$ and $\limsup_{n\to\infty}\mu_n\big(\omega\sm f^{-1}[A]\big)>0$, which would contradict condition (2) of Theorem \ref{theorem:main_theorem_BJNP_prob}. 
Theorem \ref{theorem:main_theorem_BJNP_prob} implies that $N_F$ has the BJNP.

(3) Using (1), we may again assume that $G\neq Fr$. Let $f\colon\omega\to\omega$ be a surjection such that $F\sub f(G)$. Assume that $S_G$ has the JNP and let $\seqn{\mu_n}$ be a sequence of finitely supported probability measures on $S_G$ as in Corollary \ref{cor:sf_jnp_char}. Let $\psi\colon S_G\to S_F$ be a continuous surjection defined for every $x\in S_G$ as follows (cf.\ Proposition \ref{prop:ordering_sf}.(1)):
\[\psi(x)=\big\{A\in\aA_F\colon f^{-1}[A]\in x\big\}.\]
For each $n\io$ we define the measure $\mu_n'$ on $S_F$ by the formula:
\[\mu_n'=\sum_{x\in\supp(\mu_n)}\mu_n(\{x\})\cdot\delta_{\psi(x)},\]
and proceed similarly as in (2).
\end{proof}

Let us note that it may happen that $F\le_{RK}G$ and $N_G$ does not have the BJNP or it has the BJNP but not the JNP, but $N_F$ still has the JNP. Indeed, recall that $N_F$ for $F=Fr$ has the JNP and notice that if $G$ is meager, then $Fr\le_{RK}G$ by Talagrand's characterization of meager filters (\cite[Theorem 21]{Tal80}), so it suffices to take such meager $G$ that $N_G$ does not have the BJNP (see Example \ref{example: Fsigma, nonBJNP}) or it has the BJNP but not the JNP (consider e.g. the filter $F_d$, described in the next section). 

\section{The Kat\v{e}tov preorder and complexity of filters $F$ whose spaces $N_F$ have the (B)JNP\label{sec:complexity}}

In this section we will investigate the issue for which filters $F$ the space $N_F$ has the JNP or the BJNP from the point of view of the Kat\v{e}tov preorder. 
%
%
We start with the JNP, that is, we first prove Theorem \ref{theorem:main_theorem_JNP} from Introduction, as the situation appears here to be quite immediate.

\begin{proof}[\textbf{Proof of Theorem \ref{theorem:main_theorem_JNP}}]
Let $F$ be a free filter on $\omega$. Equivalence (1)$\Leftrightarrow$(2) is proved in Theorem \ref{theorem:nf_jnp}. As a consequence of this equivalence, Proposition \ref{prop:nf_convseq_char} implies that $N_F$ has the JNP if and only if the dual ideal $F^*$ is not tall. Since it is well-known that the dual ideal $F^*$ is not tall if and only if $F$ is Kat\v{e}tov equivalent to $Fr$, we obtain equivalence (1)$\Leftrightarrow$(3).
\end{proof}

Let us recall some standard definitions concerning ideals and submeasures on $\omega$. A function $\varphi\colon\wo\to[0,+\infty]$ is \textit{a submeasure} if $\varphi(\emptyset)=0$, $\varphi(\{n\})<\infty$ for every $n\io$, and $\varphi(A)\le\varphi(A\cup B)\le\varphi(A)+\varphi(B)$ for every $A,B\in\wo$. Every non-negative measure $\mu$ on $\omega$ is a submeasure. A submeasure $\varphi$ is \textit{finite} if $\varphi(\omega)<\infty$, and is \textit{lower semi-continuous} (\textit{lsc}) if $\varphi(A)=\lim_{n\to\infty}\varphi(A\cap[0,n])$ for every $A\in\wo$. For a lsc submeasure $\varphi$ we set the following two standard ideals:
\[\Fin(\varphi)=\big\{A\in\wo\colon\ \varphi(A)<\infty\big\}\]
and
\[\Exh(\varphi)=\big\{A\in\wo\colon\ \lim_{n\to\infty}\varphi(A\sm[0,n])=0\big\},\]
called \textit{the finite ideal} and \textit{the exhaustive ideal} of $\varphi$, respectively. Trivially,
\[Fin\sub\Exh(\varphi)\sub\Fin(\varphi).\]
One can also easily show that $\Exh(\varphi)$ is an $\fsd$ P-ideal and $\Fin(\varphi)$ is an $\fs$ ideal (see \cite[Lemma 1.2.2]{Far00}). Of course, the dual filters $\Fin(\varphi)^*$ and $\Exh(\varphi)^*$ have the same Borel complexity as their dual ideals (and are free).

As we said in the proof of Theorem \ref{theorem:main_theorem_JNP}, for a given free filter $F$ on $\omega$, the space $N_F$ has the JNP if and only if the dual ideal $F^*$ is not tall. Note that in the case of ideals of the form $\Exh(\varphi)$ we have an easy characterization of tallness: the ideal $\Exh(\varphi)$ is tall if and only if $\lim_{n\to\infty}\varphi(\{n\})=0$.

\begin{proposition}\label{prop:nf_jnp_exh_tall}
Let $\varphi$ be a lsc submeasure on $\omega$. The following are equivalent:
\begin{enumerate}
	\item $N_{\Exh(\varphi)^*}$ has the JNP,
	\item $\Exh(\varphi)$ is not tall,
	\item $\limsup_{n\to\infty}\varphi(\{n\})>0$.\noproof
\end{enumerate}
\end{proposition}

For two submeasures $\varphi$ and $\psi$ on $\omega$ we write $\psi\le\varphi$ if $\psi(A)\le\varphi(A)$ for every $A\in\wo$. Following Farah \cite[page 21]{Far00}, we call a submeasure $\varphi$ \textit{non-pathological} if for every $A\in\wo$ we have:
\[\varphi(A)=\sup\big\{\mu(A)\colon\ \mu\text{ is a non-negative measure on }\omega\text{ such that }\mu\le\varphi\big\}.\]
An ideal $I$ on $\omega$ is \textit{non-pathological} if $I=\Exh(\varphi)$ for some non-pathological lsc submeasure $\varphi$. Note that in this case the formula $\psi=\min(\varphi,1)$ defines a finite non-pathological lsc submeasure such that $I=\Exh(\varphi)=\Exh(\psi)$. 
An ideal $I$ is \textit{pathological} if it is not non-pathological.  
For various characterizations of non-pathological ideals, see \cite[Corollary 5.26]{Hru11} and \cite[Theorem 5.4]{BNFP}.

Density ideals constitute an important subclass of non-pathological ideals. Recall that a submeasure $\varphi$ on $\omega$ is \textit{a density submeasure} if there exists a sequence $\seqn{\mu_n}$ of finitely supported non-negative measures on $\omega$ with disjoints supports such that:
\[\varphi=\sup_{n\io}\mu_n.\]
Clearly, $\varphi$ is a non-pathological lsc submeasure.
An ideal $I$ on $\omega$ is \textit{a density ideal} if there is a density submeasure $\varphi$ such that $I=\Exh(\varphi)$. 
For basic information concerning density ideals, see \cite[Section 1.13]{Far00}.

We are ready to prove the main theorem of this section providing a characterization of those free filters $F$ on $\omega$ for which the spaces $N_F$ have the BJNP. 

\begin{theorem}
\label{theorem:nf_bjnp_nonpath}
Let $F$ be a free filter on $\omega$. Then, the following are equivalent:
\begin{enumerate}
	\item $N_F$ has the BJNP,
	\item there is a density ideal $I$ on $\omega$ such that $F\sub I^*$,
	\item there is a non-pathological ideal $I$ on $\omega$ such that $F\sub I^*$.
\end{enumerate}
\end{theorem}
\begin{proof}
(1)$\Rightarrow$(2)
Assume that $N_F$ has the BJNP.
Let $\seqn{\mu_n}$ be a disjointly supported sequence of finitely supported probability measures on $N_F$ like in Theorem \ref{theorem:main_theorem_BJNP_prob}, i.e., in particular, 
\[\tag{$*$}\lim_{n\to\infty}\mu_n(\omega\sm A)=0$ for every $A\in F.\]

For each $A\in\wo$ define:
\[\varphi(A)=\sup_{n\io}\mu_n(A).\]
Then, $\varphi$ is a density submeasure.

We show that $F\sub\Exh(\varphi)^*$. Let $A\in F$. Since the supports of $\mu_n$'s are pairwise disjoint, for a fixed $N\in \omega$, we have $(\omega\sm [0,k])\cap \bigcup_{n < N} \supp\big(\mu_n\big) = \emptyset$ for sufficiently big $k\io$. Hence, for such $k$, we obtain
\[ \varphi\big((\omega\sm A)\sm[0,k]\big) = \sup_{n\ge N} \mu_n\big((\omega\sm A)\sm[0,k]\big)
\le \sup_{n\ge N} \mu_n(\omega\sm A). \]
Therefore, by ($*$), we conclude that
\[\lim_{k\to\infty}\varphi\big((\omega\sm A)\sm[0,k]\big)=0,\]
which proves that $\omega\sm A\in\Exh(\varphi)$ and so that $A\in\Exh(\varphi)^*$. 

\medskip

(2)$\Rightarrow$(3) Obvious.

\medskip

(3)$\Rightarrow$(1) Assume that there is a non-pathological lsc submeasure $\varphi$ such that $F$ is contained in the dual filter $\Exh(\varphi)^*$. Without loss of generality we may assume that $\varphi$ is finite (see the sentence after the definition of a non-pathological ideal). Put $I=\Exh(\varphi)$.  Since the inclusion $F\sub I^*$ implies the relation $F\le_K I^*$, by Proposition \ref{prop:rk_bjnp}.(2), it is enough to prove that $N_{I^*}$ has the BJNP.

Set
\[\alpha=\lim_{n\to\infty}\varphi(\omega\sm[0,n]),\]
and note that $\alpha>0$ (since otherwise $I=\wo$) and that $\alpha\le\varphi(\omega)<\infty$ (since $\varphi$ is finite). By the monotonicity of $\varphi$, for every $n\io$ we have:
\[\tag{$**$}\varphi(\omega\sm[0,n])>\alpha/2.\]
Put $n_0=0$. Since $\varphi$ is lower semi-continuous, there exists $n_1>n_0$ such that $\varphi\big(\big[n_0,n_1\big)\big)>\alpha/2$.
By ($**$) we have $\varphi\big(\omega\sm\big[0,n_1\big)\big)>\alpha/2$, so again, by the lower semi-continuity of $\varphi$, there is $n_2>n_1$ such that $\varphi\big(\big[n_1,n_2\big)\big)>\alpha/2$. We continue in this way until we get a strictly increasing sequence $\seqk{n_k\io}$ satisfying for every $k\io$ the inequality
\[\varphi\big(\big[n_k,n_{k+1}\big)\big)>\alpha/2.\]

The submeasure $\varphi$ is non-pathological, so for each $k\io$ there exists a non-negative measure $\mu_k$ on $\omega$ such that $\mu_k\le\varphi$, $\supp\big(\mu_k\big)\sub\big[n_k,n_{k+1}\big)$, and
\[\mu_k\big(\big[n_k,n_{k+1}\big)\big)>\alpha/4.\]
Note that $\alpha/4<\big\|\mu_k\big\|<\infty$ and set $\nu_k=\mu_k\big/\big\|\mu_k\big\|$. The function $\nu_k$ is a finitely supported probability measure on $N_{I^*}$ such that $p_{I^*}\not\in\supp\big(\nu_k\big)$.

We claim that the sequence $\seqk{\nu_k}$ satisfies, for every $A\in I^*$, the equality $\lim_{k\to\infty}\nu_k(A)=1$, and thus, by Theorem \ref{theorem:main_theorem_BJNP_prob}, the space $N_{I^*}$ has the BJNP. Let $B\in I$ and $\eps>0$. Since $\lim_{n\to\infty}\varphi(B\sm[0,n])=0$, there is $M\io$ such that $\varphi(B\sm[0,M])<\eps$. Let $k\io$ be such that $n_k>M$. It holds:
\[\nu_k(B)=\mu_k(B)\big/\big\|\mu_k\big\|=\mu_k\big(B\cap\big[n_k,n_{k+1}\big)\big)\big/\big\|\mu_k\big\|\le\varphi\big(B\cap\big[n_k,n_{k+1}\big)\big)\big/\big\|\mu_k\big\|\le\]
\[\varphi\big(B\sm[0,M]\big)\big/\big\|\mu_k\big\|<4\eps/\alpha.\]
It follows that $\lim_{n\to\infty}\nu_k(B)=0$ for every $B\in I$, hence $\lim_{n\to\infty}\nu_k(A)=1$ for every $A\in I^*$.
\end{proof}

Theorem \ref{theorem:nf_bjnp_nonpath} has several immediate consequences.
%
%
%

\begin{corollary}\label{cor:density_bjnp}
If $I$ is a density ideal, then $N_{I^*}$ has the BJNP.\noproof
\end{corollary}

\begin{corollary}\label{cor:nonpath_sub_dens}
Every non-pathological ideal is contained in some density ideal.\noproof
\end{corollary}

Recall that for every submeasure $\varphi$ on $\omega$ the ideal $\Exh(\varphi)$ is an $\F_{\sd}$ P-ideal (in particular, it is a Borel subset of $\Cantor$). 

\begin{corollary}\label{cor:nf_bjnp_fsd}
Let $F$ be a free filter on $\omega$. If $N_F$ has the BJNP, then there exists an $\F_{\sd}$ P-filter $G$ on $\omega$ such that $F\sub G$ and $N_G$ has the BJNP, too.\noproof
\end{corollary}

Recall that by classical theorems of Talagrand (\cite[Theorem 21]{Tal80}) and Sierpiński (\cite{Sie38}, cf. Bartoszy\'{n}ski \cite[Theorem 1.1]{Bar92}) every ideal of the form $\Exh(\varphi)$ for some lsc submeasure $\varphi$ is meager and of measure zero (as it is a Borel subset of $\Cantor$).
%

\begin{corollary}\label{cor:nf_bjnp_meager_meas0}
Let $F$ be a free filter on $\omega$. If $N_F$ has the BJNP, then $F$ is meager and of measure zero.\noproof
\end{corollary}

It is well-known that every free ultrafilter on $\omega$ is non-meager.

\begin{corollary}\label{cor:ultrafilter_no_bjnp}
If $F$ is a free ultrafilter on $\omega$, then $N_F$ does not have the BJNP.\noproof
\end{corollary}

Note that the last corollary also follows from Propositions \ref{prop:nf_bjnp_sf_jnp} and \ref{prop:sf_bo_ultrafilter} and Lemma \ref{lemma:basic_jnp}.(2). 

\medskip

\begin{example}\label{example:idealZ}
A prototypical ideal for the class of density ideals is \textit{the (asymptotic) density ideal} $\zZ=\Exh\big(\varphi_d\big)$, where \textit{the asymptotic density submeasure $\varphi_d$} is defined for every $A\in\wo$ as follows:
\[\varphi_d(A)=\sup_{n\io}\frac{\big|A\cap\big[2^n,2^{n+1}\big)\big|}{2^n}.\]One can show that for every $A\in\wo$ the following equivalence holds:
\[A\in\Exh\big(\varphi_d\big)\quad\Longleftrightarrow\quad\limsup_{n\to\infty}\frac{\big|A\cap[0,n)\big|}{n}=0.\]
Set $F_d=\zZ^*$. The BJNP and JNP of the spaces $N_{F_d}$ and $S_{F_d}$ were already studied in \cite[Section 4]{BKS1}, \cite[Example 4.2]{KMSZ}, and \cite[Section 6.1]{MSZ}, where it was, among other things, proven that:
\begin{itemize}
	\item $S_{F_d}$ has no non-trivial convergent sequences but contains a copy of $\bo$;
	\item $S_{F_d}$ has the JNP and there is a JN-sequence $\seqn{\mu_n}$ on $S_{F_d}$ such that $\supp\big(\mu_n\big)\sub N_{F_d}$ for every $n\io$;
	\item every JN-sequence $\seqn{\mu_n}$ on $S_{F_d}$ satisfies the equality $\lim_{n\to\infty}\big|\supp\big(\mu_n\big)\big|=\infty$;
	\item $N_{F_d}$ has the BJNP but not the JNP.
\end{itemize}
Note that most of the above facts may be easily deduced from more general results presented in this paper. 
 In particular, the last fact follows from Theorem \ref{theorem:nf_bjnp_nonpath} (or Corollary \ref{cor:density_bjnp}) and Proposition \ref{prop:nf_jnp_exh_tall}.
\end{example}

Hern\'{a}ndez-Hern\'{a}ndez and Hru\v{s}\'{a}k (\cite[Proposition 3.6]{HHH07}) observed that every density ideal is Kat\v{e}tov below $\zZ$. This, together with Theorem \ref{theorem:nf_bjnp_nonpath}, yields the following characterization of spaces $N_F$ with the BJNP, analogous to Theorem \ref{theorem:main_theorem_JNP}.

\begin{theorem}\label{theorem:nf_bjnp_z}
Let $F$ be a free filter on $\omega$. Then, the space $N_F$ has the BJNP if and only if $F\le_K\zZ^*$.
\end{theorem}
\begin{proof}
If $N_F$ has the BJNP, then by Theorem \ref{theorem:nf_bjnp_nonpath} there exists a density ideal $I$ such that $F\sub I^*$, so in particular $F\le_K I^*$. By \cite[Proposition 3.6]{HHH07} $I^*\le_K\zZ^*$, hence $F\le_K\zZ^*$.

If on the other hand $F\le_K\zZ^*$, then by Proposition \ref{prop:rk_bjnp}.(2) $N_F$ has the BJNP, since by Corollary \ref{cor:density_bjnp} $N_{\zZ^*}$ does.
\end{proof}


\subsection{Summable ideals and large families of non-homeomorphic spaces $N_F$\label{sec:summable}}

Yet another class of non-pathological ideals is constituted by summable ideals. Recall that an ideal $I$ on $\omega$ is \textit{summable} if there exists a  function $f\colon\omega\to[0,\infty)$ such that for the non-pathological lsc submeasure $\varphi_f$ defined, for every $A\in\wo$, by the formula:
\[\varphi_f(A)=\sum_{n\in A}f(n),\]
we have $I=\Exh\big(\varphi_f\big)$. 
Notice that for every $A\in\wo$ it holds:
\[A\in\Exh\big(\varphi_f\big)\quad\Longleftrightarrow\quad\sum_{n\in A}f(n)<\infty,\]
so $\Exh\big(\varphi_f\big)=\Fin\big(\varphi_f\big)$. In particular, every summable ideal is an $\fs$ P-ideal.

The following corollary follows immediately from Theorem \ref{theorem:nf_bjnp_nonpath} and the fact that summable ideals are non-pathological.

\begin{corollary}\label{cor:summable_bjnp}
If $I$ is a summable ideal, then $N_{I^*}$ has the BJNP.\noproof
\end{corollary}

Guzm\'{a}n-Gonz\'{a}lez and Meza-Alc\'{a}ntara \cite[Theorem 1]{GGMA16} proved that there exists an order-preserving embedding $\Phi$ of the poset $(\wo/Fin,\subseteq^*)$ into the set $\Sigma$ of all summable ideals endowed with the Kat\v{e}tov preordering $\le_K$. Consequently, summable ideals can be used to construct rich families of pairwise non-isomorphic filters $F$ on $\omega$ whose spaces $N_F$ have the JNP or the BJNP but not the JNP, that is, they can be used to justify Corollary \ref{cor:main_corollary_continuum} from Introduction.

\begin{proof}[\textbf{Proof of Corollary \ref{cor:main_corollary_continuum}}]
Let $\mathfrak{F}$ be a family of size $\frakc$ of pairwise $\subseteq^*$-incomparable elements of $\wo/Fin$. Set $\fF_2=\big\{I^*\colon\ I\in\Phi[\mathfrak{F}]\big\}$, where $\Phi$ is the above-mentioned order-preserving embedding of $\wo/Fin$ into $\Sigma$. Corollary \ref{cor:summable_bjnp} implies that for each $F\in\fF_2$ the space $N_F$ has the BJNP but, by Theorem \ref{theorem:main_theorem_JNP} or Proposition \ref{prop:nf_jnp_exh_tall}, it does not have the JNP. Condition (B) is thus satisfied.

We put  $\fF_1 = \big\{F\oplus Fr\colon F\in \fF_2\big\}$ (cf. Section \ref{sec:free_sums}).
It is clear that each $G\in \fF_1$ is an $\fs$ P-filter such that  $N_G$ has the JNP (Proposition \ref{prop:product_filters}). By Lemma \ref{lemma:fr_f1_f2_iso}, no two members of $\fF_1$ are isomorphic.
\end{proof}

\begin{remark}
Of course, any two distinct filters from the family $\fF_2$ are even Kat\v{e}tov incomparable. Since $F\equiv_K Fr$ for any $F\in\fF_1$, any two filters from $\fF_1$ are Kat\v{e}tov comparable (even equivalent).
\end{remark}

Dropping the assumption concerning the descriptive complexity of filters in Corollary \ref{cor:main_corollary_continuum}, we may obtain families of cardinality $2^{\frakc}$ containing pairwise non-isomorphic filters on $\omega$ whose spaces $N_F$ have the BJNP but not the JNP, etc., hence verifying Corollary \ref{cor:main_corollary_2continuum} from Introduction.

\begin{proof}[\textbf{Proof of Corollary \ref{cor:main_corollary_2continuum}}]
The corollary can be justified using the well-known fact that the Kat\v{e}tov equivalence class of any ideal contains $2^\frakc$ many pairwise non-isomorphic filters. For the reader's convenience we present below yet another elementary argument.

Let $\mathcal{G}_5$  be the family of all ultrafilters on $\omega$ and let $F$ be a fixed element of the family $\fF_2$ from Corollary \ref{cor:main_corollary_continuum}. We put:
\[\mathcal{G}_3= \big\{G\oplus Fr\colon G\in \mathcal{G}_5\big\}\quad\text{ and}\quad\mathcal{G}_4 = \big\{G\oplus F\colon G\in \mathcal{G}_5\big\}.\]
Using Corollary \ref{cor:ultrafilter_no_bjnp} and Proposition \ref{prop:product_filters}, one can easily check that, for any $i=3,4,5$ and $H\in \mathcal{G}_i$, the space $N_H$ has the property of members of $\fF_i$ declared in conditions (a)--(c). Since $|\mathcal{G}_5| = 2^{\frakc}$ and each family of pairwise isomorphic filters has the cardinality bounded by the continuum, for each $i=3,4,5$ we can select a subfamily $\fF_i\subseteq \mathcal{G}_i$ consisting of $2^{\frakc}$ many pairwise non-isomorphic filters.
\end{proof}

\begin{remark}\label{remark:two_to_cont_many_nonhomeo}
Proposition \ref{prop:ordering_nf}.(2) implies that for $i=1,\dots,5$ and each pair of distinct filters $F,G$ belonging to the family $\fF_i$ from Corollaries \ref{cor:main_corollary_continuum} and \ref{cor:main_corollary_2continuum} their spaces $N_F$ and $N_G$ are not homeomorphic.
\end{remark}

\subsection{An $\fs$ P-filter $F$ such that $N_F$ does not have the BJNP\label{sec:fs_no_bjnp}}

In Theorem \ref{theorem:nf_bjnp_nonpath} we proved that any free filter $F$ on $\omega$ whose space $N_F$ has the BJNP is contained in the filter dual to some non-pathological ideal, in particular, in some $\fsd$ P-filter (cf. also Corollary \ref{cor:nf_bjnp_fsd}). Note that not every $\fsd$-filter or even $\fs$-filter has this property. Namely, Solecki \cite{Sol00} constructed an $\fs$-ideal which cannot be extended to a non-pathological ideal. Solecki's ideal is not a P-ideal, but Filip\'ow and Szuca \cite[Example 3.6]{FS10} found an example of an $\fs$ P-ideal which cannot be covered by a summable ideal and recently it was proved by Filip\'ow and Tryba \cite[Theorem 4.12]{FT19} that the same ideal cannot be covered by a non-pathological ideal. 

Using different tools than those in \cite{FS10} and \cite{FT19}, we present below yet another example of an $\fs$ P-ideal which cannot be covered by a non-pathological ideal. Note that this construction is related to the earlier results of Mazur \cite{Maz91} who constructed an $\fs$ ideal which is not contained in any summable ideal and of Farah \cite[Sections 1.9 and 1.11]{Far00} who provided an example of an $\fs$ P-ideal which is pathological (and hence not summable)\footnote{It must be noted that the general idea standing behind the constructions from \cite{Maz91}, \cite{Far00}, and \cite{FS10} as well as from our Example \ref{example: Fsigma, nonBJNP} is quite similar.}.

In our argument we will use the following result of Herer and Christensen \cite[Theorem 1]{HCh}.

\begin{theorem}[Herer--Christensen]\label{thm_HCh}
	For each $\varepsilon> 0$, there exist a finite set $X$ and a submeasure $s$ defined on the algebra $\wp(X)$ such that $s(X)=1$ and any non-negative measure $\mu$ defined on $\wp(X)$ and dominated by $s$ (i.e., $\mu\le s$) satisfies $\mu(X)<\varepsilon$.\noproof
\end{theorem}

\begin{lemma}\label{HCh_sub}
	Let $\varepsilon> 0$ and let $s$ be a submeasure on a set $X$ given by Theorem \ref{thm_HCh}. Then, for any non-negative finite measure $\mu$ on $\wp(X)$, there exists a set $A\subseteq X$ such that $\mu(A)\ge \mu(X)/2$ and $s(A)\le 2\varepsilon$.
\end{lemma}

\begin{proof}
	Put $a= \mu(X)$, without loss of generality we can assume that $a> 0$. Observe that by the properties of $s$, the measure $(\varepsilon/a)\mu$ is not dominated by $s$, so there is $B\subseteq X$ such that $\mu(B) > (a/\varepsilon)s(B)$.
	Let $A$ be a maximal with respect to the inclusion subset of $X$ satisfying
	\[\mu(A) > \big(a/(2\varepsilon)\big)s(A).\]
	By the maximality of $A$, for every $B\subseteq X\sm A$, we have:
	\[\label{sub1}\tag{$*$}\mu(B) \le \big(a/(2\varepsilon)\big)s(B).\]
	Let $b=\mu(A)$. We will show that $b\ge a/2$. Suppose to the contrary that $b < a/2$, then $a-b = \mu(X\setminus A) > 0$. Hence, we can consider a measure $\nu$ on $\wp(X)$ defined for every $C\subseteq X$ by the formula
	\[\nu(C) = \big(\varepsilon/(a-b)\big)\mu(C\setminus A).\]
	We have $\nu(X) = \varepsilon$, therefore $\nu$ is not dominated by $s$, so there is a set $B\subseteq X$ such that
	\[\label{sub2}\tag{$**$}\nu(B) > s(B).\]
	Since
	\[\nu(B\sm A)=\nu(B)>s(B)\ge s(B\setminus A),\]
	we can assume that $B\subseteq X\setminus A$. Using the definition of $\nu$ and combining inequalities (\ref{sub1}) and (\ref{sub2}), we obtain:
	\[\big(\varepsilon/(a-b)\big)\cdot\big(a/(2\varepsilon)\big)s(B)\ge \big(\varepsilon/(a-b)\big)\cdot\mu(B) > s(B),\]
	which, after routine simplifications of the outer sides, leads to the inequality $2b > a$, contradicting our assumption on $b$.
	
	It remains to observe that the inequality $\mu(A) > \big(a/(2\varepsilon)\big)s(A)$ together with $a \ge \mu(A)$ gives the desired estimate $s(A)\le 2\varepsilon$.
\end{proof}

\begin{example}\label{example: Fsigma, nonBJNP}
	\textit{There exists a free $\fs$ P-filter $F$ on $\omega$ such that the space $N_F$ does not have the BJNP.}
\end{example}

\begin{proof}
For each $n\io$, let $A_n$ be a finite set and $s_n$ be a submeasure on $\wp(A_n)$ given by Theorem \ref{thm_HCh} applied for $\varepsilon = 2^{-n}$. Without loss of generality we can assume that the sets $A_n$ are pairwise disjoint subsets of $\omega$ and that $\omega=\bigcup_{n\io}A_n$. For every $A\in\wo$ set:
\[\varphi(A)=\sum_{n\io}s_n\big(A_n\cap A\big),\]
and notice that $\varphi$ is a lsc submeasure on $\omega$ such that $\Fin(\varphi)=\Exh(\varphi)$. Put $I=\Fin(\varphi)$ and $F=I^*$. Of course, $F$ is a free $\fs$ P-filter. 
	
We will show that the space $N_F$ does not have the BJNP. So assume that it has the BJNP and let $\seqk{\mu_k}$ be a disjointly supported sequence of finitely supported probability measures on $N_F$ as in Theorem \ref{theorem:main_theorem_BJNP_prob}. For each $k\io$, set $B_k = \supp\big(\mu_k\big)$, the sets $B_k$ are pairwise disjoint. Passing to a subsequence of $\seqk{\mu_k}$, if necessary, we can also require that each set $A_n$ intersects at most one set $B_k$. Let thus $P$ be a set of all those $n\io$ for which there exists (a unique) $k(n)\io$ such that $A_n\cap B_{k(n)}\ne \emptyset$. For every $n\in P$, put $C_n = A_n\cap B_{k(n)}$.

Applying Lemma \ref{HCh_sub} for the submeasure $s_n$ and the measure $\mu_{k(n)}\rstr C_{n}$, $n\in P$, we can find a subset $D_{n}$ of $C_{n}$ such that 
\[\mu_{k(n)}\big(D_{n}\big)\ge \mu_{k(n)}\big(C_{n}\big)/2\]
and
\[s_n\big(D_n\big)\le 2\cdot 2^{-n} =  2^{-n+1}.\]

Let $D = \bigcup_{n\in P} D_n$. For each $k\io$, put $F_k=\big\{n\in P\colon\ k=k(n)\big\}$ and note that $B_k=\bigcup\big\{C_n\colon n\in F_k\big\}$. From the  inequality $s_n(D_n)\le 2^{-n+1}$ for $n\in P$, we conclude that $D$ belongs to the ideal $I$. On the other hand, for each $k\io$, we have:
\[\mu_{k}(D) = \sum_{n\in F_k}\mu_k\big(D_n\big)\ge\sum_{n\in F_k}\mu_k\big(C_n\big)/2=(1/2)\mu_k\big(B_k\big) = 1/2,\]
a contradiction with the condition that $\lim_{k\to\infty}\mu_k(D)=0$.
\end{proof}

We do not know if one can construct a family of size continuum consisting of pairwise non-isomorphic $\fs$ P-filters $F$ such that the space $N_F$  does not have the BJNP (cf. Corollary \ref{cor:main_corollary_continuum}).


\subsection{Spaces $S_F$ without BJN-sequences on $N_F$ but with convergent sequences in $S_F^*$}

As an application of Theorem \ref{theorem:nf_bjnp_z}, in this short final part of Section \ref{sec:complexity} we observe that, for every sequence $\seqn{F_n}$ of free filters based on some sequence $\seqn{A_n}$, the space $N_{\sum F_n}$ does not have the BJNP.

Recall that the Fubini product $Fr\otimes Fr$ is the free filter on $\omega\times\omega$ given by the formula:
\[Fr\otimes Fr=\Big\{A\sub\omega\times\omega\colon \big\{n\io\colon \big\{k\io\colon (n,k)\in A\big\}\in Fr\big\}\in Fr\Big\}.\]
As in the case of free sums, we can think of $Fr\otimes Fr$ as of a free filter on $\omega$.

Fix a sequence $\seqn{F_n}$ of free filters based on some sequence $\seqn{A_n}$. It is immediate that $Fr\otimes Fr\le_K\sum F_n$. However, as observed in \cite[page 7]{Hru17}, no $\fsd$ filter on $\omega$ is Kat\v{e}tov above $Fr\otimes Fr$. In particular, we have $Fr\otimes Fr\not\le_K\zZ^*$, and so, by Proposition \ref{prop:rk_bjnp} and Theorem  \ref{theorem:nf_bjnp_z}, the space $N_{\sum F_n}$ does not have the BJNP. On the other hand, Lemma \ref{lemma:sf_f_convseq}.(2) implies that $S_{\sum F_n}^*$ contains a non-trivial convergent sequence. We thus get the following result.

\begin{corollary}\label{cor:sf_jnp_nf_no_bjnp}
For every sequence $\seqn{F_n}$ of free filters based on some sequence $\seqn{A_n}$, the space $S_{\sum F_n}^*$ contains a non-trivial convergent sequence (in particular, the space $S_{\sum F_n}$ has the JNP) but the space $N_{\sum F_n}$ does not have the BJNP.
\noproof
\end{corollary}

Since $Fr\otimes Fr$ can be treated as a filter of the form $\sum F_n$ for some sequence $\seqn{F_n}$ of free (Fr\'echet) filters, the above corollary also applies to it. This way, as $Fr\otimes Fr$ is an $\mathbb{F}_{\sigma\delta\sigma}$ set, it is an example of a Borel filter $F$ such that $S_F^*$ contains a non-trivial convergent sequence but the space $N_F$ does not have the BJNP.

Corollary \ref{cor:sf_jnp_nf_no_bjnp} implies that in order to study which spaces among those of the form $S_F^*$ have the JNP, we need to apply completely different techniques and tools than we described in this section.
 
\section{General Tychonoff spaces with the JNP and the BJNP \label{sec:tychonoff}}

Let $X$ be a space and assume that $\seqn{x_n}$ is a non-trivial sequence in $X$ convergent to some point $p\in X\sm\big\{x_n\colon n\io\big\}$. By Lemma \ref{lemma:basic_jnp}.(1), $X$ has the JNP. This consequence can also be seen from a more general point of view. Namely, put $Z=\big\{x_n\colon\ n\io\big\}\cup\{p\}$ and endow it with the subspace topology. Obviously, in $Z$, the sequence $\seqn{x_n}$ is still convergent to $p$. 
For each open neighborhood $U$ of $p$ in $Z$, let $A(U)=\big\{n\io\colon\ x_n\in U\big\}$, and set $F=\big\{A(U)\colon U\text{ is an open neighborhood of }p\text{ in }Z\big\}$. It is immediate that $F=Fr$ and that the space $N_{Fr}$ is homeomorphic to $Z$. Hence, $X$ contains a subspace homeomorphic to some space $N_G$, where $G$ is a free filter on $\omega$, which has the JNP, and hence $X$ itself must have the JNP. In this section we study a counterpart of this situation for the property BJNP. This, together with results from the previous sections, will allow us to immediately obtain various sufficient conditions for $X$ to have the BJNP or the JNP.

We need to introduce a piece of notation. Fix a space $X$, its infinite countable subset $Y$, and a point $x\in\ol{Y}\sm Y$. Set $Z_X(x,Y)=Y\cup\{x\}$ and endow it with the subspace topology. By $\mathfrak{N}_X(x)$ we denote the neighborhood system of $x$ in $X$, that is, the collection of all (not necessarily open) subsets $U$ of $X$ such that $x\in\intt U$. We then put:
\[\mathfrak{F}_X(x,Y)=\big\{U\cap Y\colon U\in\mathfrak{N}_X(x)\big\}.\]
Note that since $x\not\in Y$ and $Y$ is infinite countable, $\mathfrak{F}_X(x,Y)$ is a free filter on the \textit{set} $Y$. If $f\colon\omega\to Y$ is a bijection, then $F=\big\{f^{-1}[V]\colon\ V\in\mathfrak{F}_X(x,Y)\big\}$ is a free filter on $\omega$---we will say in this case that $F$ is $f$-\textit{associated} (or, shortly, \textit{associated}, if $f$ is not important) \textit{to} $\mathfrak{F}_X(x,Y)$. The bijection $f$ gives rise to the bijective continuous function $\varphi_f\colon N_F\to Z_X(x,Y)$ such that $\varphi_f\rstr\omega=f$ and $\varphi_f\big(p_F\big)=x$. 

\begin{theorem}\label{theorem:tychonoff_x_bjnp}
Suppose $X$ is a space. Let $Y$ be its countable subset and $x\in\ol{Y}\sm Y$. Let $f\colon\omega\to Y$ be a bijection and $F$ a free filter on $\omega$ $f$-associated to $\mathfrak{F}_X(x,Y)$.

Then, the space $N_F$ has the BJNP if and only if $X$ admits a BJN-sequence $\seqn{\mu_n}$ such that $\supp\big(\mu_n\big)\sub Z_X(x,Y)$ for every $n\io$ and $\lim_{n\to\infty}\big\|\mu_n\rstr(Y\sm V)\|=0$ for every $V\in\mathfrak{F}_X(x,Y)$. 

In particular, if $N_F$ has the BJNP, then $X$ has the BJNP.
\end{theorem}
\begin{proof}

Assume first that $N_F$ has the BJNP. Let $\seqn{\nu_n}$ be a BJN-sequence of measures on $N_F$. 
For each $n\io$ define the measure $\mu_n$ on $X$ by the formula
\[\mu_n(B)=\nu_n\big(\varphi_f^{-1}\big[B\cap Z_X(x,Y)\big]\big),\]
where $B$ is a Borel subset of $X$; it follows that $\big\|\mu_n\big\|=1$ and that $\supp\big(\mu_n\big)=\varphi_f\big[\supp\big(\nu_n\big)\big]$ is finite and contained in $Z_X(x,Y)$. 

Since $\varphi_f$ is continuous, given any $g\in C_p^*(X)$, we get that $g\circ\varphi_f\in C_p^*\big(N_F\big)$, hence $\mu_n(g)=\nu_n\big(g\circ \varphi_f\big)\to 0$ as $n\to\infty$. Consequently, $\seqn{\mu_n}$ is a BJN-sequence on $X$ and so $X$ has the BJNP. Finally, if $V\in\mathfrak{F}_X(x,Y)$, then $f^{-1}[V]\in F$, so by Proposition \ref{prop:nf_bjnseq_char}.(3) it holds:
\[\lim_{n\to\infty}\big\|\mu_n\rstr(Y\sm V)\big\|=\lim_{n\to\infty}\big\|\nu_n\rstr\big(\omega\sm f^{-1}[V]\big)\big\|=0.\]

Assume now that $X$ admits a BJN-sequence $\seqn{\mu_n}$ of measures such that $\supp\big(\mu_n\big)\sub Z_X(x,Y)$ for every $n\io$ and $\lim_{n\to\infty}\big\|\mu_n\rstr(Y\sm V)\|=0$ for every $V\in\mathfrak{F}_X(x,Y)$. For each $n\io$ and $z\in N_F$, set $\nu_n(\{z\})=\mu_n\big(\varphi_f(\{z\})\big)$. This defines a finitely supported measure on $N_F$ such that $\big\|\nu_n\big\|=1$. Also, for every $A\in F$, we have:
\[\lim_{n\to\infty}\big\|\nu_n\rstr(\omega\sm A)\big\|=\lim_{n\to\infty}\big\|\mu_n\rstr\big(Y\sm f[A]\big)\big\|=0,\]
since $f[A]\in\mathfrak{F}_X(x,Y)$. Put $P_n=\big\{x\in\supp\big(\nu_n\big)\colon \nu_n(\{x\})>0\big\}$ and $N_n=\supp\big(\nu_n\big)\sm P_n$. Since $\seqn{\mu_n}$ is a BJN-sequence, we get (cf. \cite[Lemma 3.1]{MSZ}):
\[\lim_{n\to\infty}\big\|\nu_n\rstr P_n\big\|=\lim_{n\to\infty}\big\|\nu_n\rstr N_n\big\|=1/2.\]
Appealing to Proposition \ref{prop:nf_bjnseq_char}, we learn that $\seqn{\nu_n}$ is a BJN-sequence on $N_F$ and hence $N_F$ has the BJNP. 
%
\end{proof}

Let us note that, without loss of generality, we may require that the sequence $\seqn{\mu_n}$ in Theorem \ref{theorem:tychonoff_x_bjnp} is disjointly supported (by Lemma \ref{lemma:basic_jnp}.(3)). Also, applying Theorem \ref{theorem:main_theorem_BJNP_prob} and the methods similar to those used in the proof of Theorem \ref{theorem:tychonoff_x_bjnp}, one can obtain the following sufficient condition for a space $X$ to have the BJNP.

\begin{corollary}\label{cor:tychonoff_x_prob_meas}
Suppose $X$ is a space. Let $Y$ be its countable subset and $x\in\ol{Y}\sm Y$. Let $F$ be a free filter on $\omega$ associated to $\mathfrak{F}_X(x,Y)$.

Then, the space $N_F$ has the BJNP if and only if $X$ admits a (disjointly supported) sequence $\seqn{\mu_n}$ of finitely supported probability measures such that  $\supp\big(\mu_n\big)\sub Y$ for every $n\io$ and $\lim_{n\to\infty}\mu_n(V)=1$ for every $V\in\mathfrak{F}_X(x,Y)$.

In particular, if $X$ admits a sequence $\seqn{\mu_n}$ as above, then $X$ has the BJNP.\noproof
\end{corollary}

The next corollary applies in particular to the situation when the given mapping $\varphi$ is an injection or an embedding.
\begin{corollary}\label{cor:cont_nf_x}
Suppose $X$ is a space. Assume that $G$ is a free filter on $\omega$. Let $\varphi\colon N_G\to X$ be a continuous function such that $\varphi^{-1}\big(\varphi\big(p_F\big)\big)=\big\{p_F\big\}$. Then, if $N_G$ has the BJNP, then $X$ has the BJNP, too.
\end{corollary}
\begin{proof}
Assume that $N_G$ has the BJNP. Put $Y=\varphi[\omega]$ and $x=\varphi\big(p_F\big)$. It follows that $Y$ is an infinite countable subset of $X$ and that $x\in\ol{Y}\sm Y$. Let $F$ be a filter on $\omega$ associated to $\mathfrak{F}_X(x,Y)$. Since $\varphi$ is continuous, $F$ is Kat\v{e}tov below $G$ (cf. the discussion before Proposition \ref{prop:ordering_nf}). By Proposition \ref{prop:rk_bjnp}.(2), $N_F$ has the BJNP, so by Theorem \ref{theorem:tychonoff_x_bjnp} $X$ has the BJNP, too.
\end{proof}

The following result is a consequence of Corollary \ref{cor:main_corollary_2continuum}.(b), Remark \ref{remark:two_to_cont_many_nonhomeo}, and Corollary \ref{cor:cont_nf_x}.

\begin{corollary}\label{cor:continuum_many_nonhomeomorphic}
There exists a family $\fF$ of $2^\frakc$ many pairwise non-homeomorphic countable infinite spaces with exactly one non-isolated point and without any non-trivial convergent sequences, and such that if any $Y\in\fF$ homeomorphically embeds into a space $X$, then $X$ has the BJNP. \noproof
\end{corollary}

The next corollary follows from Theorem \ref{theorem:nf_bjnp_z} and Corollary \ref{cor:cont_nf_x}.

\begin{corollary}\label{cor:cont_nf_x_fr_z}
Let $X$, $G$, and $\varphi$ be as in Corollary \ref{cor:cont_nf_x}. 
If $G\le_K \zZ^*$, then $X$ has the BJNP.\noproof
\end{corollary}

Theorem \ref{theorem:tychonoff_x_bjnp}, its counterpart for the JNP (which can be justified in a similar way), and their consequences described above can be used to obtain sufficient conditions implying that for a given space $X$ the spaces $C_p(X)$ and $C_p^*(X)$ contain complemented copies of the space $(c_0)_p$, as well as that for a given compact space $K$ the Banach space $C(K)$ does not have the Grothendieck property, and hence to prove Corollaries \ref{cor:corollary_C}--\ref{cor:corollary_F} from Introduction. We will need here two Important Results:
\begin{enumerate}[({I}R.1)]
	\item if $X$ is an infinite space, then the space $C_p(X)$ (resp. $C_p^*(X)$) contains a complemented copy of $(c_0)_p$ if and only if $X$ has the JNP (resp. the BJNP) (see \cite[Theorem 1]{BKS1} and \cite[Theorem 4.4]{KMSZ}), and 
	\item if $K$ is an infinite compact space, then the space $C(K)$ does not have the $\ell_1$-Grothendieck property if and only if $K$ has the JNP (see \cite[Theorem 4.1]{KSZgroth}).
\end{enumerate}
Recall that the $\ell_1$-Grothendieck property is a natural weakening of the Grothendieck property, which can be described as follows: for a compact space $K$, the space $C(K)$ has \textit{the $\ell_1$-Grothendieck property} if, for every sequence $\seqn{\mu_n}$ of finitely supported Radon measures on $K$ such that $\mu_n(f)\to 0$ for every $f\in C(K)$, we also have $\mu_n(B)\to0$ for every Borel subset $B\sub K$. Obviously, the lack of the $\ell_1$-Grothendieck property implies the lack of the Grothendieck property, but the converse is false (see \cite[Section 5]{KSZgroth}).

\begin{corollary}\label{cor:c0p_compl}
Let $X$ be a space. If there is a countable subset $Y\sub X$ and a point $x\in\ol{Y}\sm Y$ such that, for a free filter $F$ on $\omega$ associated to the filter $\mathfrak{F}_X(x,Y)$, the space $N_F$ has the JNP (resp. the BJNP), then the space $C_p(X)$ (resp. $C_p^*(X)$) contains a complemented copy of the space $(c_0)_p$. \noproof
\end{corollary}

\begin{corollary}\label{cor:ell1_gr}
Let $K$ be a compact space. If there is a countable subset $Y\sub K$ and a point $x\in\ol{Y}\sm Y$ such that, for a free filter $F$ on $\omega$ associated to the filter $\mathfrak{F}_K(x,Y)$, the space $N_F$ has the BJNP, then the Banach space $C(K)$ does not have the $\ell_1$-Grothendieck property and hence it does not have the Grothendieck property. \noproof
\end{corollary}

%

\begin{proof}[\textbf{Proof of Corollaries \ref{cor:corollary_C}--\ref{cor:corollary_F}}] 
The corollaries follow immediately from results (IR.1) and (IR.2) and Corollaries \ref{cor:continuum_many_nonhomeomorphic} and \ref{cor:cont_nf_x_fr_z}.
\end{proof}

%


It is natural to ask whether Theorem \ref{theorem:tychonoff_x_bjnp} can be strengthened by dropping from the right hand side of the conclusion the condition that $\lim_{n\to\infty}\big\|\mu_n\rstr(Y\sm V)\|=0$ for every $V\in\mathfrak{F}_X(x,Y)$, that is, in other words, whether it is true that a space $X$ has the BJNP if and only if there exist an infinite countable subset $Y\sub X$ and a point $x\in\ol{Y}\sm Y$ such that for a filter $F$ associated to $\mathfrak{F}_X(x,Y)$ the space $N_F$ has the BJNP. Unfortunately, as the next example shows, it is not possible.

It is also worth noting that for disjoint countable subsets $Y_1$ and $Y_2$ of a space $X$ such that $Y_2\sub\ol{Y}_1\sm Y_1$ and a point $x\in X$ such that $x\in\ol{Y_1}\sm Y_1$ and $x\in\ol{Y_2}\sm Y_2$ it may be true that the space $Z_X\big(x,Y_1\big)$ does not have the BJNP while the space $Z_X\big(x,Y_2\big)$ contains a non-trivial sequence convergent to $x$ (see Corollary \ref{cor:sf_jnp_nf_no_bjnp} for a relevant example).

\begin{example}\label{example:schachermayer}
Let $\sS$ denote the Boolean subalgebra of $\wo$ such that, for every $A\in\wo$, $A\in\sS$ if and only if there is $K\io$ such that for every $k\ge K$ we have: $2k\in A$ if and only if $2k+1\in A$. It is easy to see that the Stone space $St(\sS)$ is a compactification of $\omega$.

The algebra $\sS$ was introduced by Schachermayer in \cite[Example 4.10]{Sch82} (cf. Remark \ref{rem:schachermayer_bereznitski}), where it was proved that its Stone space does not have the Grothendieck property and that the remainder $L=St(\sS)\sm\omega$ is homeomorphic to $\ostar$. Further properties of $\sS$ are studied in \cite[Section 5.1]{MSZ} (where its Stone space $St(\sS)$ is denoted by $K_B$), in particular, it is pointed there that $St(\sS)$ admits a JN-sequence $\seqn{\mu_n}$ such that $\supp\big(\mu_n\big)\sub\omega$ and $\big|\supp\big(\mu_n\big)\big|=2$ for every $n\io$ (consider simply the measures $\mu_n=\frac{1}{2}\big(\delta_{2n}-\delta_{2n+1}\big)$). We will now briefly show that $St(\sS)$ does not contain any infinite countable subspace $Y$ and point $x\in\ol{Y}\sm Y$ such that the space $N_F$, where $F$ is a free filter on $\omega$ associated to $\mathfrak{F}_{St(\sS)}(x,Y)$, has the BJNP. So, for the sake of contradiction, let us assume that such $Y$ and $x$ exist. Of course, $x\in L$, since otherwise $x$ would be isolated. 

Set $Z=Z_{St(\sS)}(x,Y)$. 
Let $F$ be a free filter on $\omega$ associated to $\mathfrak{F}_{St(\sS)}(x,Y)$. Since $N_F$ has the BJNP, by Corollary \ref{cor:tychonoff_x_prob_meas} there is a disjointly supported sequence $\seqn{\mu_n}$ of finitely supported probability measures on $Z$ such that $\supp\big(\mu_n\big)\sub Y$ for every $n\io$ and
\[\tag{$*$}\lim_{n\to\infty}\mu_n(V)=1\]
for every $V\in\mathfrak{F}_{St(\sS)}(x,Y)$.

If $Y\cap L$ is infinite, then, by Corollary \ref{cor:tychonoff_x_prob_meas} and the fact that $Z\cap L$ does not have the BJNP (being contained in $L\cong\ostar$), it follows that $\mu_n(L)=\big\|\mu_n\rstr L\big\|\to0$ as $n\to\infty$. Thus, without loss of generality, we may assume that $Y\sub\omega$. By passing to a subsequence, we can find a strictly increasing sequence $\seqn{k_n\io}$ such that
\[\supp\big(\mu_n\big)\sub\big[2k_n,2k_{n+1}\big)\sub\omega\]
for every $n\io$.

Put:
\[E=\bigcup_{n\io}\big[2k_{2n},2k_{2n+1}\big)\quad\text{and}\quad O=\omega\sm E=\bigcup_{n\io}\big[2k_{2n+1},2k_{2n+2}\big).\]
Both $E$ and $O$ are elements of the algebra $\sS$. Since $x$ is an ultrafilter on $\sS$, either $E\in x$ or $O\in x$. Without loss of generality we may assume that $E\in x$, so $Y\cap E\in\mathfrak{F}_{St(\sS)}(x,Y)$. We then have:
\[\liminf_{n\to\infty}\mu_n(Y\cap E)=\liminf_{n\to\infty}\mu_n\big(Y\sm(Y\cap O)\big)\le\liminf_{n\to\infty}\mu_{2n+1}(\emptyset)=0,\]
which contradicts ($*$) and finishes the proof.
\end{example}

\begin{remark}\label{rem:schachermayer_bereznitski}
It is easy to see that the Stone space $St(\sS)$ does not contain any non-trivial convergent sequences. Since it admits a JN-sequence $\seqn{\mu_n}$ such that $\big|\supp\big(\mu_n\big)\big|=2$ for every $n\io$, it follows from Proposition \ref{prop:sf_convseq_supps2} that $St(\sS)$ is not homeomorphic to any space of the form $S_F$ where $F$ is a free filter on $\omega$. In particular, by Proposition \ref {prop:sf_comes_from_bo_converse}, $St(\sS)$ is not homeomorphic to any space of the form $\bo/\fF$ where $\fF$ is a non-empty closed subset of $\ostar$.

The space $St(\sS)$ can be however described in another way. Denote $\E=\{2n\colon n\io\}$ and $\O=\{2n+1\colon n\io\}$, and set $B_{\E}=\ol{\E}^{\bo}\sm\E$ and $B_{\O}=\ol{\O}^{\bo}\sm\O$. Of course, $B_{\E}\cong B_{\O}\cong\ostar$. The bijection $h\colon\E\to\O$, defined for each $n\io$ by $h(2n)=2n+1$, gives rise to the natural homeomorphism $H\colon B_{\E}\to B_{\O}$. We define an equivalence relation $R$ on $\bo$ by declaring its equivalence classes in the following way: for each $n\io$ set $[n]_R=\{n\}$, and for each $x\in B_{\E}$ set $[x]_R=\{x,H(x)\}$. Then, one can show that $St(\sS)$ is homeomorphic to the quotient space $\bo/R$.

The space $\bo/R$ was first studied by Bereznitski\u{\i} \cite{Ber71}.  Arkhangel’ski\u{\i} asked whether there exists a compact infinite space $K$ such that the space $C_p(K)$ is not linearly homeomorphic to the product $C_p(K)\times\R$, and suggested that $\bo/R$ might be such a space (see \cite[page 93]{Ark87}). In \cite{Mar97} the first author claimed without a proof that $C_p(\bo/R)$ is linearly homeomorphic to $C_p(\bo/R)\times\R$, and provided a proper example of $K$ for which the spaces $C_p(K)$ and $C_p(K)\times\R$ are not linearly homeomorphic.

We are now able to justify briefly the statement that $C_p(\bo/R)$ is linearly homeomorphic to $C_p(\bo/R)\times\R$. Since $\bo/R$ has the JNP, it follows by \cite[Theorem 1]{BKS1} that the space $C_p(\bo/R)$ contains a complemented copy of the space $(c_0)_p$, that is, there exists a closed linear subspace $Y$ of $C_p(\bo/R)$ such that $C_p(\bo/R)$ is linearly homeomorphic to the product topological vector space $Y\times(c_0)_p$. Since $(c_0)_p$ is linearly homeomorphic to $(c_0)_p\times\R$, we get that $C_p(\bo/R)$ is linearly homeomorphic to $Y\times(c_0)_p\times\R$ and hence to $C_p(\bo/R)\times\R$.
\end{remark}

\end{document}